\documentclass{amsart}
\usepackage[pdftex]{graphicx}
\usepackage{amsmath,amssymb,amsbsy,latexsym}
\usepackage{graphicx, color}

\newtheorem{theorem}{Theorem}[section]
\newtheorem{proposition}[theorem]{Proposition}
\newtheorem{corollary}[theorem]{Corollary}

\newtheorem{definition}[theorem]{Definition}

\newtheorem{remark}[theorem]{Remark}

\numberwithin{equation}{section}

\usepackage{amscd} 

\newcommand{\ali}{\arraycolsep1pt\begin{array}{rl}}  
\newcommand{\eali}{\end{array}}  

\newcommand{\RR}{{\boldsymbol{R}}}

\newcommand{\ZZ}{{\boldsymbol{Z}}}

\newcommand{\lra}{\longrightarrow}
\newcommand{\DS}{\displaystyle}
\newcommand{\op}{\mathrm}
\newcommand{\bd}{\partial}
\newcommand{\pa}{\partial}

\newcommand{\ssm}{\smallsetminus}
\newcommand{\vect}{\overrightarrow}

\newcommand{\abstracttext}{We study the shapes of compact connected 3-manifolds with connected smooth boundary in the 3-dimensional Euclidean space $\boldsymbol{R}^3$. We call them bounded domains. Since compact connected surfaces in $\boldsymbol{R}^3$ bound unique bounded domains, the objects are the same as compact connected surfaces in $\boldsymbol{R}^3$. To understand their shapes, we use the Morse height functions $F: M\longrightarrow \boldsymbol{R}$ which are the orthogonal projections from the bounded domains $M$ to lines, and their Reeb graphs $\mathcal{R}_F$ and $\mathcal{R}_{F|\partial M}$ which are obtained by identifying connected components of level sets of maps to points. We introduce the weighted Reeb graphs $\mathcal{R}_F^w$ and the weighted indexed Reeb graphs $\mathcal{R}_F^{wi}$ which give certain good understanding of the bounded domains $M$. We investigate whether a bounded domain admits a simple Morse height function $F:M\lra \RR$, for example, with the weighted Reeb graphs $\mathcal{R}_F^w$ with small weight. We show that if the weights are less than 2. $M$ can be deformed by isotopy to an embedded handle body. The original question which lead us to investigate bounded domains is the following question: \lq\lq{}Can the domain $M$ be isotoped so that, for every point of the boundary $\partial M$, there is a ray from the point which intersects the domain $M$ only at the end point ?\rq\rq\ In other words, \lq\lq{}Can $M$ be isotoped to $\iota(M)$ so that every point of $\bd \iota (M)$ is visible from the infinity ?\rq\rq\ Under the minNCP hypothesis, we show that if a bounded domain $M$ can be isotoped to $\iota(M)$ so that every point of the boundary is visible from the infinity, then $M$ is an embedded handlebody, that is, $M$ is a regular neigtborhood of a spacial graph.  Here the minNCP hypothesis asserts that, if $M$ is isotopic to a visible $\iota(M)$, $\iota(M)$ can be taken so that $z:\iota(M)\lra \RR$ is a Morse height function with minimum number of critical points in the isotopy class of the embedding $M\subset \RR^3$.}

\begin{document}

\title{Bounded domains in the 3-dimensional space}
\author{Takashi Tsuboi}
\address{Takashi Tsuboi, University of Tokyo, Meguro-ku Komaba 3-8-1, Tokyo 153-8914, Japan. }
\email{tsuboi@ms.u-tokyo.ac.jp}
%
%

\maketitle

\begin{abstract}
\abstracttext
\end{abstract}

\section*{Introduction}
\label{sec:0}

In this article, we begin with giving a survey on the shapes of bounded domains in the 3-dimensional Euclidean space. 

As we feel living in the space-time and things are usually thought as 3-dimensional objects, our data on such 3-dimensional objects may be given as compact 3-dimensional submanifolds with smooth boundary in the 3-dimensional Euclidean space which we call bounded domains.
Three-dimensional manifold theory was developed considerably in the 20th century and we know already that there are very strong restriction on such objects. It is an interesting task now to review to what extent we understand compact 3-bodies. Hence we first planned to explain several basic facts known for  bounded domains in the 3-dimensional Euclidean space to those who are not specialized in mathematics, and to give several applications of such results. However, in the preparation of this article we found many interesting problems on bounded domains and these problems seems to have not yet been settled. Thus this article contains several problems which might interest mathematicians working in low dimensional topology.

To describe  bounded domains $M$ in 3-dimensional Euclidean space $\RR^3$, we look at height functions $F:M\lra \RR$ which are the orthogonal projections to lines and whose restrictions to their boundaries $\bd M$ are the Morse functions in Section \ref{sec:Morse_height_function}.
 We call such $F:M\lra \RR$ the Morse height functions.

For a Morse height function $F:M\lra \RR$, we consider the Reeb graphs 
$\mathcal{R}_{F|\bd M}$ and $\mathcal{R}_F$ which are obtained from $\bd M$ and $M$ by identifying each connected component of a fiber of $F|\bd M$ and $F$ to a point. Then we have the map $\mathcal{R}_{F|\bd M}\lra \mathcal{R}_F$ which contains information on the shape of $M$. From this map we define the weighted Reeb graph $\mathcal{R}^w_F$ and the weighted indexed Reeb graph $\mathcal{R}^{wi}_F$ which might be more convenient to describe the shapes of bounded domains in Section \ref{sec:weighted_Reeb_Graph}.  We show that bounded domains with the Morse height functions whose the weighted Reeb graphs have small weights has simple structure. Namely, if the weights of the weighted Reeb graph $\mathcal{R}^w_F$ are less than 2, $M$ is diffeomorphic to handlebodies. 

This consideration lead us to reply to the following question: 
\lq\lq{}Can the domain $M$ be isotoped to $\iota(M)$ so that, for every point of the boundary $\partial \iota(M)$, there is a ray from the point which intersects the domain $\iota(M)$ only at the end point ?\rq\rq\ In other words, \lq\lq{}Can $M$ be isotoped to $\iota(M)$ so that every point of $\bd \iota (M)$ is visible from the infinity ?\rq\rq\ Under the minNCP hypothesis, 
we show in Section \ref{sec:visibility}, that if a bounded domain $M$ can be isotoped to $\iota(M)$ so that every point of the boundary is visible from the infinity, then $M$ is an embedded handlebody, that is, $M$ is a regular neigtborhood of a spacial graph.  Here the minNCP hypothesis asserts that, if $M$ is isotopic to a visible $\iota(M)$, $\iota(M)$ can be taken so that $z:\iota(M)\lra \RR$ is a Morse height function with minimum number of critical points in the isotopy class of the embedding $M\subset \RR^3$.

Osamu Saeki informed us that the Morse functions on a closed 2-dimensional manifold which bounds a 3-dimensional manifolds and their Reeb graphs are treated by Carlos Curley \cite{Curley} and by Ken-ichi Iwamoto \cite{Iwamoto}. They studied whether a Morse function on a closed 2-dimensional manifold can extend a nonsingular function on the 3-dimensional manifold. Their papers contain a number of illustrating figures. Note that the notations there is a little different from ours.

\section{Bounded domains in the 3-dimensional Euclidean space}
\label{sec:1}

First we define the object which we are interested in.

\begin{definition}[Bounded domains]
A bounded domain $M$ in the 3-dimensional Euclidean space $\RR^3$
is a compact connected 3-dimensional manifold with smooth connected boundary embedded in $\RR^3$. 
\end{definition}

For a bounded domain $M$, a point $p\in M$ whose $\varepsilon$-neighborhood $\mathcal{N}_\varepsilon(p)$ for some positive real number $\varepsilon$ is contained in $M$ is called an {\it interior point} and the interior points of $M$ form an open subset of $\RR^3$, $\op{Int}(M)$, the interior of $M$. 
At a point $p$ on the boundary $\bd M=M\ssm \op{Int}(M)$ of a compact 3-dimensional manifold $M$, we have a $C^\infty$ function $f$ defined in an $\varepsilon$-neighborhood $\mathcal{N}_\varepsilon(p)$ of $p$, such that 
$$\ali f(p)&=0\,,\\\op{grad}(f)_{(p)}&=\left(\dfrac{\pa f}{\pa x}, \dfrac{\pa f}{\pa y},\dfrac{\pa f}{\pa z}\right)^T_{(p)}\neq\vect{0}\qquad\text{and}\\[4mm]
 M\cap \mathcal{N}_\varepsilon(p)&=\left\{x\in \mathcal{N}_\varepsilon(p)\ \big|\ f(x)\leqq0\right\}.\eali$$
In fact, for a bounded domain $M$, there is a $C^\infty$ function $f$ defined on $\RR^3$, such that  $\op{grad}(f)_{(p)}\neq \vect{0}$ for $p\in f^{-1}(0)$, 
$$ 
M=\{p\in \RR^3\ \big|\ f(p)\leqq 0\} \text{\ \ and\ \ } 
\bd M=\{p\in \RR^3\ \big|\ f(p)= 0\}.$$ 
For a point $p$ on the boundary $\bd M$, $\op{grad}(f)_{(p)}$ is an outward normal vector at $p\in \bd M$. 

For a bounded domain $M$, the boundary surface $S=\bd M=M\ssm \op{Int}(M)$ is a smooth 2-dimensional manifold. Since at any boundary point $p$ we have the normal vector $\op{grad}(f)_{(p)}$ to the boundary surface $S$ pointing outwards, the boundary surface $S$ is oriented.
It is well-known that compact connected oriented surfaces $S$ are classified by the Euler characteristic number $\chi(S)$ which is expressed by the genus $g(S)$ of $S$ as $\chi(S)=2-2g(S)$ and $\chi(S)$ takes the values $2$, $0$, $-2$, $\dots$, the even integers less than or equal to 2. The surface is diffeomorphic to the 2-sphere $S^2$, the 2-torus $T^2$, the oriented surface $\varSigma_2$, $\dots$, accordingly.

By the way, a closed connected surface in $\RR^3$ is always orientable. 
For, if there were a non-orientable embedded surface $S$, we can find a closed path $c:[0,1]\lra S$ on the surface $S$ such that, for the curve $c([0,1])$ continuously pushed out from $S$ to one side of $S$, $c(0)$ and $c(1)$ are on the different side. Joining these two points, we obtain a closed curve $C$  intersecting the surface transversely only once, by the intersection theory with $\ZZ/2\ZZ$ coefficients, existence of such an intersection of a closed surface and a closed curve implies that the third homology group of the total space $\RR^3$ is nontrivial. For, in terms of the cup product $\cup$, $H^2(S;\ZZ/2\ZZ)\cong \ZZ/2\ZZ$, $H^1(C;\ZZ/2\ZZ)\cong \ZZ/2\ZZ$ and the only one intersection point shows that
$$\cup :H^2(S;\ZZ/2\ZZ)\times H^1(C;\ZZ/2\ZZ) \lra H^3(\RR^3;\ZZ/2\ZZ)$$
is the nontrivial map. Since $H^3(\RR^3;\ZZ/2\ZZ)= 0$, we have a contradiction.

A closed connected surface $S$ divides the total space $\RR^3$ into two domains, one is bounded and the other is not bounded. This fact can be shown by the same reasoning as the previous paragraph. For, if it does not divide $\RR^3$, the complement $\RR^3\ssm S$ is path-connected and
two points near a point of the surface $S$ are conneced by a path. Joining the two point we obtain a closed curve $C$ intersecting $S$ once, which yields a contradiction. Note that $S$ divides the total space $\RR^3$ into at most two components. 
Since $S$ is compact, $S$ is contained in a ball of large radius, one of the components contains the exterior of the ball, the other component is contained in the ball and thus bounded. 

If we have a finite number of compact connected surfaces in $\RR^3$ which are disjoint, by assigning the domain bounded by the surface to each surface, and we will have the well-defined inclusion relation on the bounded domains and we have the partial order of the connected surfaces. 
The boundary of a compact connected 3-manifold in the Euclidean 3-space $\RR^3$ can be non-connected. But its boundary surfaces 
has a simple partial order. One of the boundary surfaces is the exterior surface which bounds the domain containing all other surfaces and the domains bounded by the others are disjoint.
Thus it is important to understand compact connected 3-manifolds with
\textit{connected} boundary and it is equivalent to understand how a connected orientable surface $S$ embeds in $\RR^3$.

To look at surfaces in the Euclidian 3-space, it is better to think the surface is embedded in $S^3$,  the one-point compactification $\RR^3\cup\{\infty\}$ of $\RR^3$ first, and then think about the position of the point $\infty$.
Here $$S^3=\{(x_0,x_1,x_2,x_3)\in \RR^4\ \big| \ x_0{}^2+x_1{}^2+x_2{}^2+x_3{}^2=1\}.$$  There is the stereographic projection from the south pole $(-1,0,0,0)\in S^3$ which maps 
$$S^3\ni (x_0,x_1,x_2,x_3)\longmapsto \big( \dfrac{x_1}{x_0+1},\dfrac{x_2}{x_0+1},\dfrac{x_3}{x_0+1}\big)\in \RR^3,$$
where $\infty$ corresponds to the south pole $(-1,0,0,0)\in S^3$.

\section{Surfaces in the 3-dimensional sphere}
\label{sec:2}

There are famous important theorems on surfaces in the 3-dimensional sphere. 

\begin{theorem}[Sch\"oflies problem, Alexander theorem\ (\cite{Alexander})]\label{th:Alexander}
A smooth 2-sphere in the 3-dimensional sphere $S^3$ divides $S^3$ into two domains both of which are  3-disks $D^3$.
\end{theorem}

\begin{remark}
The Alexander horned sphere (\cite{AlexanderHS}) is a topologically embedded 2-sphere in $S^3$ which bounds non simply connected domains. This is the reason why we look at compact 3-manifolds with smooth boundary.
\end{remark}

\begin{theorem}[Solid torus theorem]\label{th:solid_torus_theorem}
A smooth 2-torus in the 3-dimensional sphere $S^3$ divides $S^3$ into two domains and one of which is a solid torus diffeomorphic to $D^2\times S^1$.
\end{theorem}

It is worth reviewing the proof of Theorem \ref{th:solid_torus_theorem}, which we follow the argument in the Rolfsen's book (\cite{Rolfsen}).

The proof uses the description of the fundamental group. 
Let the torus $T^2$ divides $S^3$ into two domains $A$ and $B$: $$S^3=A\cup B\quad\text{and}\quad A\cap B=T^2.$$

Then one of the two inclusion maps $T^2\subset A$ and $T^2\subset B$ induces a homomorphism in the fundamental groups with nontrivial kernel. 

For,
assume that both of the inclusion maps $T^2\subset A$ and $T^2\subset B$ induce injections in the fundamental groups. Then by taking a base point $b\in T^2$, the van Kampen theorem asserts that the fundamental group $\pi_1(S^3,b)$ is written as the amalgamated product of the groups $\pi_1(A,b)$ and  $\pi_1(B,b)$ over $\pi_1(T^2,b)$:
$$\pi_1(S^3,b)\cong \pi_1(A,b)\underset{\pi_1(T^2,b)}{*} \pi_1(B,b).$$
This  amalgamated product is defined to be the group generated by the generators of $\pi_1(A,b)$ and $\pi_1(B,b)$ with the relations of the groups $\pi_1(A,b)$ and $\pi_1(B,b)$ and the relations which say the images of generators of $\pi_1(T^2,b)$ in $\pi_1(A,b)$ and $\pi_1(B,b)$ coincide. By the assumption of the injectivity of the induced maps, the amalgamated product is a non trivial group, but $\pi_1(S^3,b)$ is trivial. 

Thus one of the two inclusion maps $T^2\subset A$ and $T^2\subset B$ induces a homomorphism in the fundamental groups with nontrivial kernel, and 
we may assume that $\pi_1(T^2,b)\lra \pi_1(A,b)$ has nontrivial kernel, for which we can apply the loop theorem (\cite{Papa2}). 

\begin{theorem}[The loop theorem]\label{th:loop_theorem}
For a compact 3-dimensional manifold $M$ with boundary $\bd M$, if the kernel $\op{ker}(i_*:\pi_1(\bd M,b)\lra \pi_1(M,b))$ is nontrivial, then there is a \textbf{simple} closed curve (a simple loop) $c$ on $\bd M$ such that $[c]\neq \mathbf{1}\in \op{ker}(i_*)$. 
\end{theorem}
Then we apply Dehn's lemma to this simple closed curve $c$.
\begin{theorem}[Dehn's lemma for a boundary simple closed curve]\label{th:Dehn}
For a compact 3-dimensional manifold $M$ with boundary, if  a simple closed curve $c$ on $\bd M$ represents a nontrivial element in $\op{ker}(i_*:\pi_1(\bd M,b)\lra \pi_1(M,b))$, then there is an embedded disk $D^2\subset M$ such that $c =\bd D^2$. This disk is called a compressing disk.
\end{theorem}

\textit{Proof} {\textit{of Theorem \ref{th:solid_torus_theorem}.}}
In the case of an embedded $T^2$ in $S^3$, $T^2$ divides $S^3$ into two domains $A$ and $B$, and we may assume that $\op{ker}(i_*:\pi_1(T^2,b)\lra \pi_1(A,b))$ is nontrivial.  Then by Theorem \ref{th:Dehn},  we find a simple closed curve $c$ embedded in $T^2$ homtopically nontrivial in $T^2$, which bounds an embedded disk $D^2$ in $A$. (Here, the loop theorem is not necessary for $T^2$ because any closed curve on  $T^2$ is homotopic to a simple closed curve. However Dehn's lemma is necessary.)

For the simple closed curve $c$ in $T^2$, the completion of $T^2\ssm c$ is an annulus. The disk $D^2$ is always two-sided, that is, $D^2\subset A$ has a product neighborhood $$D^2\times (-1,1)\subset D^2\times [-1,1]\subset A,$$ where $\bd D^2\times (-1,1)\subset T^2$ is a neighborhood of $c=\bd D^2\times \{0\}$ in $T^2$. Then $T^2\ssm \big(\bd D^2\times (-1,1)\big)$ is an annulus, and this annulus together with two disks $D^2\times \{-1\}$ and $D^2\times \{+1\}$ form a sphere. 
By Theorem \ref{th:Alexander}, this sphere in $S^3$ divides $S^3$ into two 3-disks. Then one of the two 3-disks contains $D^2\times (-1,1)$ and the other 
does not contain  $D^2\times (-1,1)$, where $D^2\times (-1,1)$ is the neighborhood of $D^2$ in $A$. We see that the 3-disk which does not contain  $D^2\times (-1,1)$ coincides with $A\ssm \big(D^2\times (-1,1)\big)$.
This implies that $A$ is diffeomorphic the union of the 3-disk and $D^2\times [-1,1]$ attached along two 2-disks $D^2\times \{-1,1\}$, hence $A$ is diffeomorphic to $D^2\times S^1$ and Theorem \ref{th:solid_torus_theorem} is proved.
\hfill \qed\\

The solid torus theorem says that the part $A$ is a tubular neighborhood of an embedded circle, that is, a knot in $S^3$.

It is very interesting to look at the shape of $B=S^3\ssm A$. The bounded domain $B$ together with $D^2\times [-1,1]$ is diffeomorphic to a 3-disk. This means $B$ is diffeomorphic to the 3-disk with $\op{Int}(D^2)\times [-1,1]$ deleted. This $\op{Int}(D^2)\times [-1,1]$ is a neighborhood of an arc $\{0\}\times [-1,1]$ in the 3-disk. This arc is the knotted part of the circle $\{0\}\times S^1\subset D^2\times S^1\cong A$, and $B$ is diffeomorphic to the knot exteior. It is called a cube with knotted hall in \cite{Rolfsen}. We use the term knot exteriors in this article.
If a knot exterior is a solid torus, then the knot is unknot and the 3-sphere is just decomposed into the union of two solid tori attached along the boundary:
$S^3=S^1\times D^2\underset{S^1\times S^1}\cup D^2\times S^1$. This case we call the knot exterior trivial.

\begin{figure}
\begin{center}
\includegraphics[width=5.5cm]{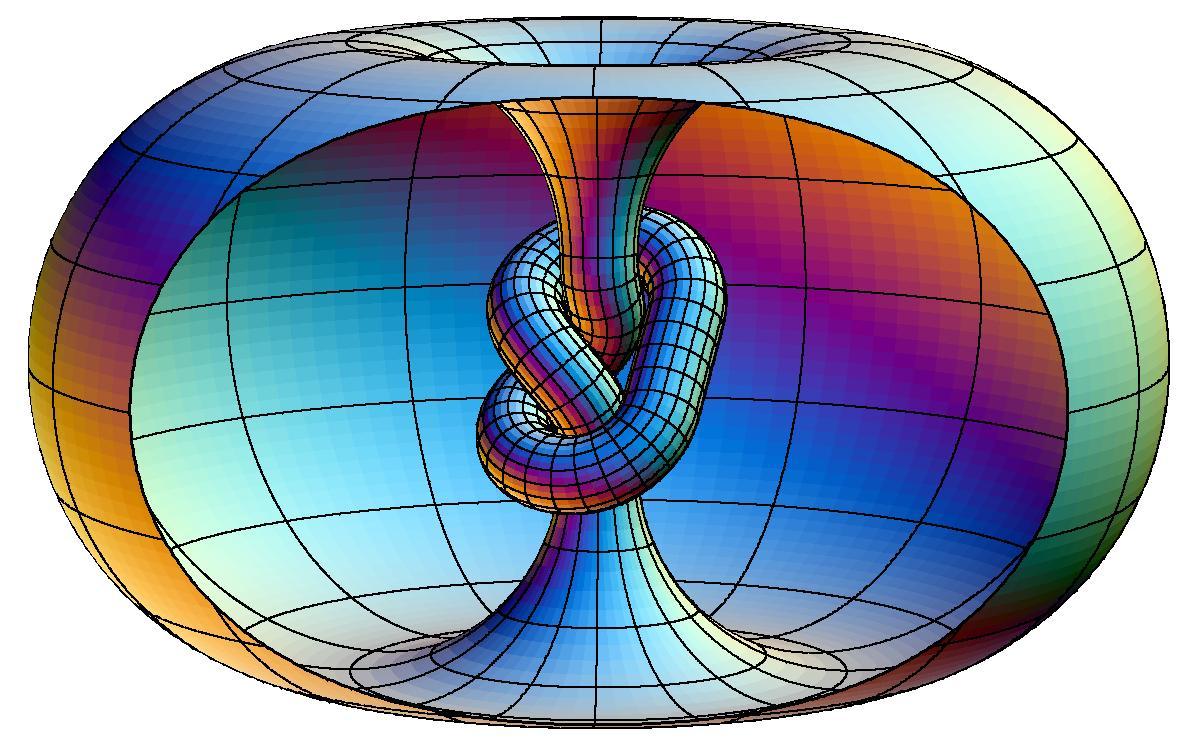}\ \ 
\includegraphics[width=5.5cm]{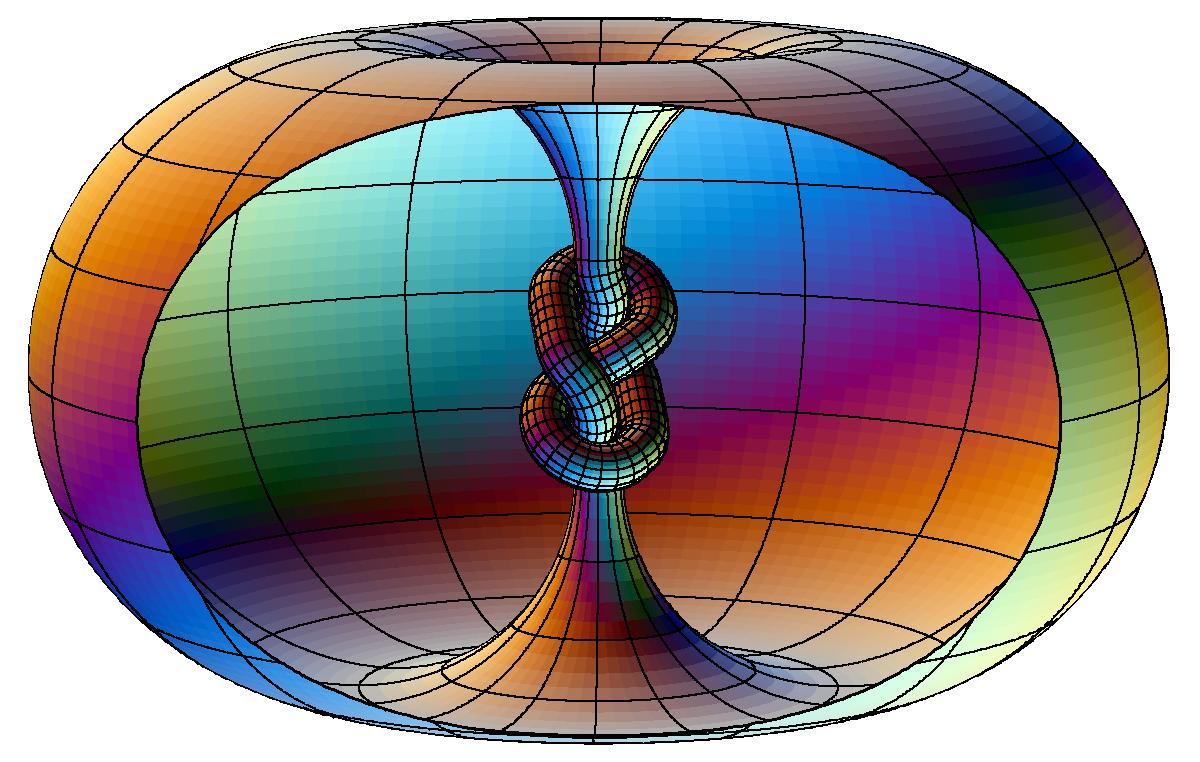}
\caption{The knot exteriors of the trefoil knot (left) and the figure eight knot (right). The left figure has been used in the cover of Sugaku Tsushin published by Mathematical Society of Japan.
}\label{fig:mathsoc_trefoil}
\end{center}
\end{figure}

\begin{proposition}
In the 3-dimensional Euclidean space $\RR^3$, an embedded torus bounds either a solid torus or a knot exterior. 
\end{proposition}

Figure \ref{fig:mathsoc_trefoil} 
 shows the knot exteriors of the trefoil knot and the figure eight knot, which were drawn by Tetsuro Kawasaki \cite{Kawasaki}. It should be understood that the boundary surfaces do not touch themselves and the windows are closed. 
Without the front windows they look very similar.

\section{Surfaces of higher genus}
\label{sec:3}

Now we look at an embedded connected surface $S$ of genus $g\geqq 2$.

A tubular neighborhood of a knot (an embedded circle in the 3-dimensional Euclidian space $\RR^3$) is diffeomorphic to the solid torus $D^2\times S^1$. 
In the same way, a regular neighborhood of a spacial graph is diffeomorphic to a handlebody. 
\begin{definition}[Handlebodies]
A handlebody of genus $g$ is a boundary connected sum 
$$\overset{g}{\overbrace{D^2 \times S^1 \natural\cdots \natural D^2 \times S^1}}$$
of $g$ solid tori $D^2 \times S^1$. 
\end{definition}

Since the circles embedded in a plane in  $\RR^3$ are un knots, the solid tori which are tubular neighborhoods of planer ciecles are called unknotted.  
In the same way, regular neighborhoods of connected planar graphs in $\RR^3$ are called unknotted handlebody. It is of genus $g$ if the Euler number of the connected  planar graph is  $1-g$ (which is equal to that of the regular neighborhood). 

The boundary $\bd M$ of an embedded handlebody $M$ of genus $g$ is a surface of genus $g$. The general situation of connected surfaces $S$ of genus $g$ embedded in the 3-dimensional Euclidean space $\RR^3$ is far  more complicated.

The surface $S$ divides $S^3=\RR^3\cup\{\infty\}$ into two domains $A$ and $B$, and as in the case of the torus, by the Loop theorem (Theorem \ref{th:loop_theorem}) and Dehn's lemma (Theorem \ref{th:Dehn}), we find a simple closed curve $c$ on $S$ homotopically non trivial on $S$ which bounds a disk $D^2$ in one domain, say $A$. Then there is a neighborhood $D^2\times (-1,1)\subset D^2\times [-1,1] \subset A$.

We look at the surface 
$$S'=(S\ssm \big(\bd D^2\times (-1,1)\big)) \cup \big(D^2\times \{-1,1\}\big),$$
which is called the surface obtaind from the surface $S$ by the surgery along the 2-disk $D^2$. See Figure \ref{fig:surgery}. 
The surface $S'$ is
is either a connected surface of genus $g-1$ or the disjoint union of connected surfaces of genus
$k\geqq 1$ and $g-k\geqq 1$, depending on whether $c$ is non-separating or separating.

In the former case, we have a connected bounded domain $A'= A\ssm D^2\times (-1,1)$ with $S'=\bd A'$ being a connected surface of genus $g-1$ and a simple curve $c'=\{0\}\times [-1,1]$ outside of $A'$ joining two points of $\bd A'$, and $A$ is the union of $A'$ and a tubular neighborhood of the curve $c'$.

In the latter case, we have two bounded domains $A_1$ with $\bd A_1$ being a connected surface of genus $k$ and $A_2$ with $\bd A_2$ being a connected surface of genus $g-k$ which are disjoint such that $$A_1\sqcup A_2= A\ssm \big(D^2\times (-1,1)\big),\quad S'=\bd A_1\sqcup \bd A_2$$  and there is a simple curve $c'$ joining a point of $\bd A_1$ and a point of $\bd A_2$, and $A$ is the union of $A_1\sqcup A_2$ and a tubular neighborhood of the curve $c'$. 

\begin{figure}
\begin{center}
\includegraphics[width=10cm]{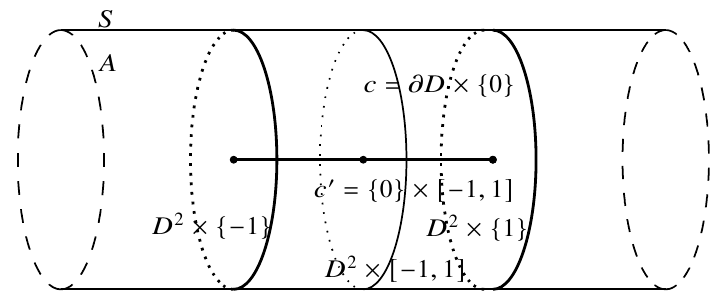}
\caption{The simple closed curve $c$ on $S$ bounds a compressing disk $D^2$ in $A$. We have a neighborhood $D^2\times [-1,1]\subset A$. 
$S'=(S\ssm \big(\bd D^2\times (-1,1)\big)) \cup \big(D^2\times \{-1,1\}\big)$.}
\label{fig:surgery}
\end{center}
\end{figure}

In both cases, we say that the bounded domain $A$ is obtained from the bounded domain $A'$ or $A_1\sqcup A_2$ by attaching a 1-handle which is a tubular neighborhood of the curve $c'$. 

We look at the case where $g=2$ to understand the situation. Let $A$ be the domain which contains the embedded 2-disk $D^2$ with $\bd D^2$ being homotopically nontrivial simple closed curve on the surface $S$ of genus 2. 
By the surgery along $D^2$, from the surface $S$ of genus 2 we obtain $S'$ which is one embedded torus or the union of two disjointly embedded tori.

In the former case, $A$ is either the union of a solid torus and $D^2\times [-1,1]$ attached along $D^2\times \{-1,1\}$ or the union of a knot exterior  and $D^2\times [-1,1]$ attached along $D^2\times \{-1,1\}$ (This situation is expressed as a solid torus or  a knot exterior with 1-handle attached). 

In the latter case, $A\ssm \big(D^2\times (-1,1)\big)$ is a union of two disjoint bounded domains each of which is either a solid torus or a knot exterior. 
There are three cases for the two solid tori bounded by the two tori:
\begin{itemize}
\item[(1)]\ The two solid tori are disjoint; 
\item[(2)]\ One solid torus contains the other solid torus; 
\item[(3)]\ The union of the two solid tori is the whole $S^3$.
\end{itemize}

In Case (1), these two solid tori are contained in the bounded domain $A$, and the bounded domain $A$ is obtained from these two solid tori (which is a tubular neighborhood of a 2-component link) by attaching a 1-handle.

In Case (2), the outer solid torus is a tubular neighborhood of a knot $k_0$ and the inner solid torus is a tubular neighborhood of a knot $k_1$ in the outer solid torus. In this case the knot $k_1\subset S^3$ is called a satelite of the knot $k_0$.
The bounded domain $A$ contains the inner solid torus and the complement of the outer solid torus, and $A$ is obtained from the union of the outer knot exterior and the inner solid torus by attaching a 1-handle. 

In Case (3), we have one knot exterior $A_1$ is contained in the other solid torus $D^2\times S^1$ which is the tubular neighborhood of another knot. We can show that 
the knot exterior $A_1$ is contained in a 3-disk in the solid torus except the case where the solid tori are the regular neighborhood of a Hopf link. 

For, let $D^2$ be a disk of the form $D^2\times \{*\}$ in the solid torus $D^2\times S^1$. We look at the intersection of the boundary torus $\bd A_1$ of the knot exterior and $D^2$. We may assume that the intersection is transverse and the number of components of the intersection is minimized among the isotopy class of $D^2$ fixing $\bd D^2$. The intersection is a union of disjoint loops on the torus $\bd A_1$. If a component of the intersection is a loop which bounds a 2-disk on the torus $\bd A_1$, then the inner-most component of the intersection in this 2-disk can be removed by an isotopy. Since we assumed the number of components is minimum, there are only essential loops in $\bd A_1$ which appear in the intersection. Now we look at the intersection in the 2-disk $D^2=D^2\times \{*\}$. then the inner-most component in $D^2$ bounds a 2-disk $E$ in $D^2$.  This means that an essential loop on $\bd A_1$ bounds a 2-disk $E$ in the solid torus $D^2\times S^1$. 

By using a neighborhood $E\times [-1,1]$ of this 2-disk $E$, we obtain a 2-sphere $$(\bd A_1 \ssm \big(\bd E\times (-1,1)\big))\cup \big(E\times \{-1,1\}\big)$$ in $D^2\times S^1$. This sphere bounds a 3-disk in $D^2\times S^1$. If this 3-disk does not contain  $E\times [-1,1]$, then we see that the solid torus bounded by $\bd A$ is contained in $D^2\times S^1$ and, if $A$ is at the same time a knot exterior, 
the solid tori are the regular neighborhood of a Hopf link. If this 3-disk contains $E\times [-1,1]$, the knot exterior $A$ is contained in a 3-disk in 
$D^2\times S^1$. Thus there are two exteriors of knots in disjoint 3-disks, and $A$ is obtained from them by attaching a 1-handle. 

As is seen even in the case where $g=2$, embedded surfaces of genus $g$ is quite complicated.

It is important however that by finding a homotopically nontrivial loop (an essential loop) on the surface of genus $g$ contractible in one of the two bounded domains obtained by dividing by $S$, and by modifying the surface $S$ by using the 2-disk (by surgery), we obtain surfaces of lower genera. After obtaining surfaces of lower genera,  we continue finding essential loops on new surfaces and surgering the surfaces. We see that the disks in bounded domains used to contract essential loops can be taken disjointly in $S^3$. In each surgery, either the number of components of $S$ increase and the genus of each component obtained by surgery decreases, or the number of components of $S$ is preserved and the genus of the surgered component decreases. 
After performing the process we obtain finitely many 2-spheres in $S^3$. 

The starting surface $S$ divides $S^3$ into two bounded domains $A$ and $B$. 
If we look at one of the two, say $A$, either we delete $D^2\times (-1,1)$ from $A$ or we add $D^2\times [-1,1]$ to $A$, to change the part $\bd D^2\times (-1,1)$ of $S$ to $D^2\times \{-1,1\}$.
When we delete $D^2\times (-1,1)$ from $A$ and the number of components of the surface increases, 
the number of component of $A$ increases.
When we add $D^2\times [-1,1]$ to $A$ and the number of components of the surface increases, the number of component of $B$ increases. 
In the final stage, the surface becomes a disjoint union of finitely many 2-spheres. The components of bounded domain obtained by dividing $S^3$ by the surface are 3-disks with a disjoint union of finitely many 3-disks deleted.

At each stage, we think of a bipartite graph with the vertices being the components of bounded domain obtained by dividing $S^3$ by the surface and the edges being given by the components of the surface where two bounded domains are adjacent.  
Note that the vertices have the origin either $A$ or $B$, and hence the graph is bipartite, and that the graph is a tree.  When the number of the components of the surface changes, the bipartite tree changes. This bipartite tree describes completely the final stage, the configuration of the disjoint 2-spheres, but the final stage may not be unique. If we have a configuration of disjoint 2-spheres in $S^3$, we have a unique bipartite tree. 

We can think of recovering the original $S$ from the disjoint union of 2-spheres in $S^3$. 
This can be performed by attaching $D^2\times [-1,1]$ along $D^2\times\{-1,1\}$ to 
 the $A$ part deleting $\op{Int}(D^2)\times [-1,1]$ from the $B$ part, or 
attaching $D^2\times [-1,1]$ along $D^2\times\{-1,1\}$ to 
 the $B$ part deleting $\op{Int}(D^2)\times [-1,1]$ from the $A$ part.
This operation is done step by step and it is quite complicated because in each step the core $\{0\}\times [-1,1]$ of $D^2\times [-1,1]$ can be extremely knotted.  

\section{Morse height functions on bounded domains}
\label{sec:Morse_height_function}

In order to describe bounded domains, it is standard to consider height functions on them. It will be used in studying the visibility of bounded domains in the section \ref{sec:visibility}.

It is known (see for example, \cite{Tsuboi} page 108) that the height function defined by the orthogonal projection to a generic line is a Morse function. Here a Morse function on the bounded domain $M\lra \RR$ is a function which has no critical points on $M$ and whose restriction to the boundary $\bd M\lra \RR$ has only non-degenerate critical points.
If the orthogonal projection to a line is a Morse function, we call this projection a Morse height function. 

We take the orthonormal coordinates $(x,y,z)$ such that the projection to the $z$-axis is a Morse height function (or we can rotate $M$ so that $z$ is a Morse height function). The inverse image of a point of the $z$-axis is a level set.
Then at the critical point $(x_0,y_0,z_0)$ of the function $z|\bd M$, the boundary surface $\bd M$ in a neighborhood $V$ of $(x_0,y_0,z_0)$ is described as the graph of a function $f$ defined in a neighborhood $U$ of $(x_0,y_0)$: $$\bd M \cap V=\{ (x,y,f(x,y))\ \big|\  (x,y)\in U\},$$
such that 
$$\big(\dfrac{\pa f}{\pa x}, \dfrac{\pa f}{\pa y}\big)_{(x_0,y_0)}=(0,0),$$ 
and the Hessian matrix $$\begin{pmatrix}\dfrac{\pa^2 f}{\pa x^2}& \dfrac{\pa^2 f}{\pa x\pa y}\\
 \dfrac{\pa^2 f}{\pa x\pa y}&\dfrac{\pa^2 f}{\pa x^2}\end{pmatrix}_{(x_0,y_0)}$$is non-degenerate, which is the definition that the critical point is non-degenerate. 
By the Morse Lemma, we can choose the local coodinates $(u_1,u_2,v)$ around $(x_0,y_0,z_0)$ such that 
$$v=f(u_1,u_2)-f(0,0)= \left\{\begin{matrix}-u_1{}^2-u_2{}^2 &\text{ (index 2)\ \ \ \ \ \ }\\
\ \ u_1{}^2-u_2{}^2 &\text{ (index 1)\ \  or }\\
\ \ u_1{}^2+u_2{}^2& \text{ (index 0),\ \ \ \ \ }
\end{matrix} \right.$$
where $u_1$ and $u_2$ are smooth functions on $x$ and $y$. Near $(x_0,y_0,z_0)$, $x$ and $y$ are also smooth functions on $u_1$ and $u_2$ and, $(u_1,u_2)\mapsto(x,y)$ is the local inverse mapping of $(x,y)\mapsto (u_1,u_2)$.
See Figure \ref{fig:contour}.

\begin{figure}
\begin{center}
\includegraphics[width=6.cm]{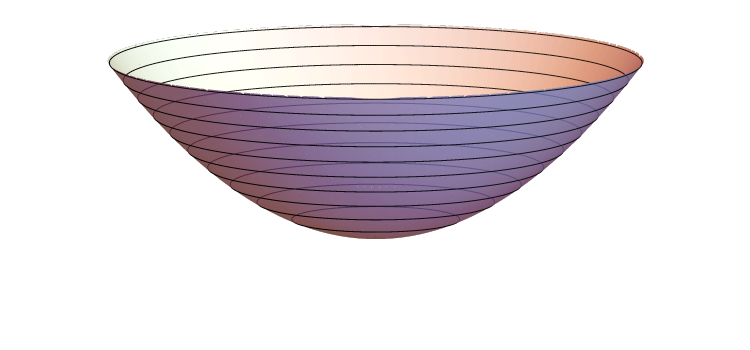} 
\includegraphics[width=6.cm]{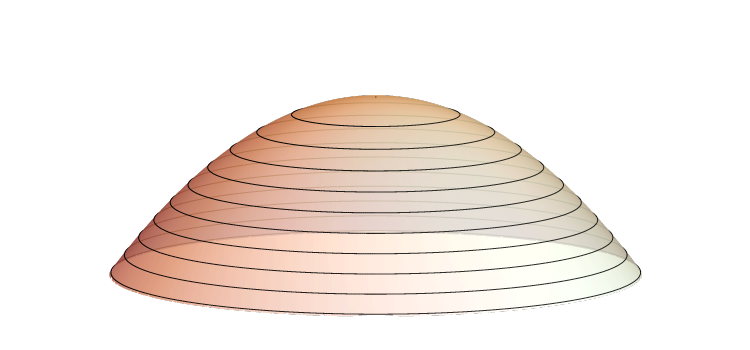}\vskip 2mm

\includegraphics[width=7.cm]{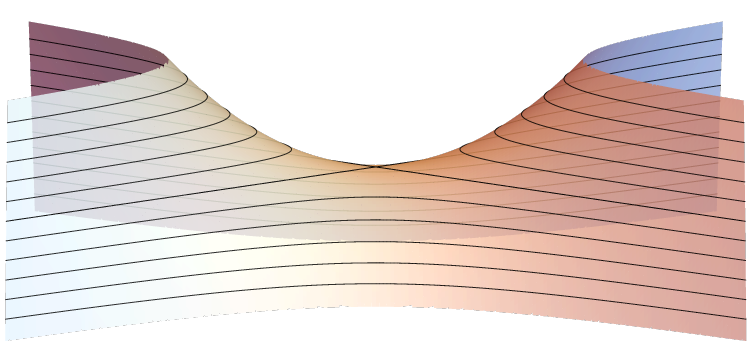}
\caption{Non degenerate critical points. The indice are 0 (upper left), 2 (upper right), 1 (lower center), respectively.}\label{fig:contour}
\end{center}
\end{figure}

For the Morse height function $F$ on $M$, we have the Reeb graphs of $F$ and $F|\bd M$. 
\begin{definition}
The Reeb graph $\mathcal{R}_F$ of $F:M\lra \RR$ is the quotient space $M/\sim $ of $M$ where an equivalence class is a connected component of the inverse image of a point. In other words, the Reeb graph $\mathcal{R}_{F}$ is obtained from $M$ by identifying each component of the level set to a point.
The Reeb graph $\mathcal{R}_{F|\bd M}$ of $F|\bd M:\bd M\lra \RR$ is defined in the same way.
\end{definition}
The following proposition helps to understand the shape of a bounded domain.
\begin{proposition}
For a bounded domain $M$ such that the projection $F$ to the $z$-axis is the Morse function, the inverse image $(F|\bd M)^{-1}(z)$ of a regular value $z\in \RR$ is a disjoint union of circles and the inverse image $F^{-1}(z)$ is the domain with boundary $(F|\bd M)^{-1}(z)$. A connected component of the inverse image $F^{-1}(z)$ is a 2-disk with the union of finitely many disjoint 2-disks in it deleted. 
The inclusion map of the boundary circles to a component of the inverse image $F^{-1}(z)$ induces $\mathcal{R}_{F|\bd M}\lra \mathcal{R}_{F}$.
\end{proposition}

\begin{proposition}\label{prop:pi1}
The inclusion $\bd M\lra M$ induces the height preserving map $\mathcal{R}_{F|\bd M}\lra \mathcal{R}_{F}$, and we have the following commutative diagram:
$$\begin{CD} \bd M @>>> M\\
@VVV  @VVV\\
\mathcal{R}_{F|\bd M}@>>> \mathcal{R}_F \end{CD} $$ 
The vertical arrows induce surjections in their fundamental groups.
$$\begin{CD} \pi_1(\bd M) @>>> \pi_1(M)\\
@VVV  @VVV\\
\pi_1(\mathcal{R}_{F|\bd M})@>>> \pi_1(\mathcal{R}_F) \end{CD} $$ 
\end{proposition}

\begin{proof}
Since the fibers of $\bd M\lra \mathcal{R}_{F|\bd M}$ and $M\lra  \mathcal{R}_F$ are path connected,  any loop on $\mathcal{R}_{F|\bd M}$ and on $\mathcal{R}_F$ homotopically lifts to a loop in $\bd M$ and in $M$, respectively and 
 the vertical arrows in the $\pi_1$ diagram are surjections.
\end{proof}

Note that if the homomorphism $\pi_1(\bd M)\lra \pi_1(M)$ has nontrivial kernel, then there is a compression disk in $M$ and $M$ can be simplified using this compressing disk as we discussed in Section 3.  It can be surjective as in the case of handlebodies. For a knot complement $M$, $\pi_1(\bd M)\cong \ZZ\times \ZZ$ and if it is not injective the knot is a trivial knot. Hence usually the homomorphism is injective. Note also that the Reeb graphs $\mathcal{R}_{F|\bd M}$ and $\mathcal{R}_F$ are 1-dimensional and hence $\pi_1(\mathcal{R}_{F|\bd M})$ and $\pi_1(\mathcal{R}_F)$ are free groups. \\

By tracing the level sets $F^{-1}(z)$ for $z\in F(M)$ from $\op{min}(F(M))$ 
to  $\op{max}(F(M))$, we see the movement of $F^{-1}(z)$ with respect to the parameter $z\in F(M)$. The topology of $F^{-1}(z)$ and $F^{-1}((-\infty,z])$ changes only when $z$ is a critical value, and when $z$ passes through a critical value, 
$F^{-1}((-\infty,z])$ changes around the critical point  $(x_0,y_0,z_0)$ as   
$f^{-1}(\{v\leqq -\varepsilon\})$ to $f^{-1}(\{v\leqq +\varepsilon\})$ for 
$v=f(u_1,u_2)=\pm u_1{}^2\pm u_2{}^2$ in the $(u_1,u_2)$ coordinates.

\begin{figure}
\begin{center}
\includegraphics[width=5cm]{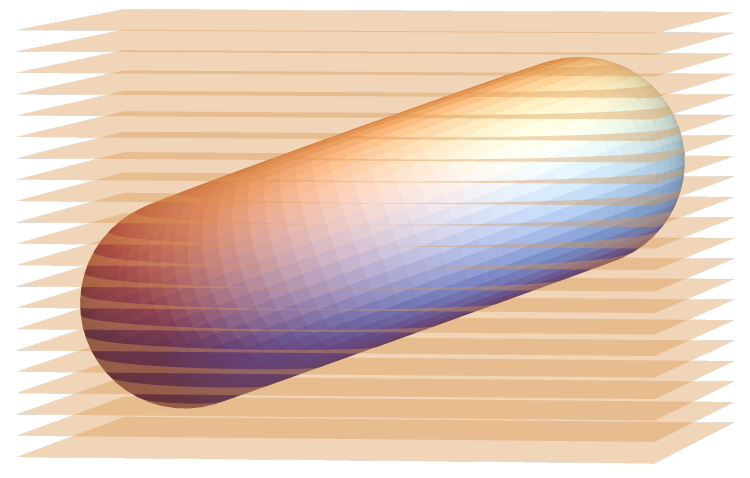}
\includegraphics[width=5cm]{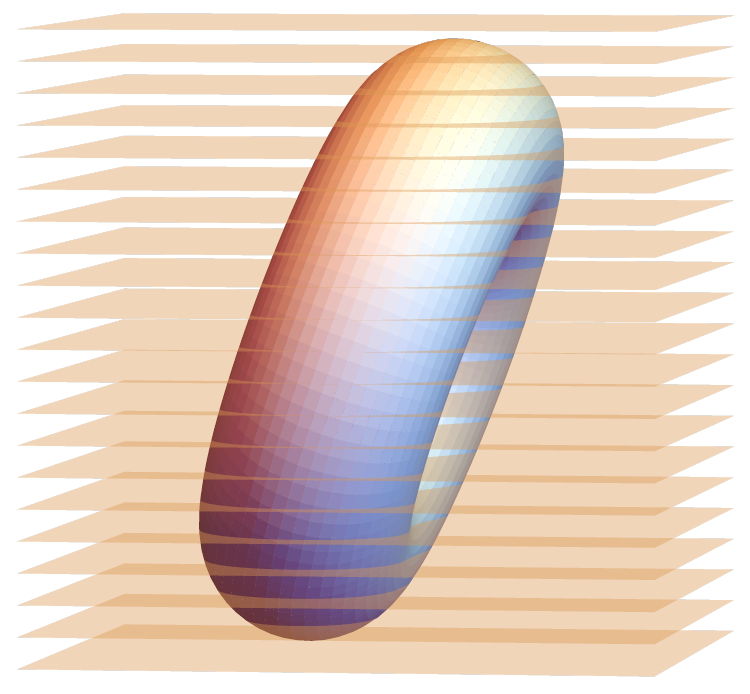}
\caption{Examples of Morse height functions on the 2-torus}\label{fig:torusPlanes}
\end{center}
\end{figure}

For example, for the solid torus $M$ which is obtained as the solid of revolution obtained from a disk in the half plane $\{(x,y,0)\in \RR^3\ \big|\ y>0\}$ by rotating around the $x$-axis, the orthogonal projection to the $z$-axis is the Morse function.
This Morse function has 4 critical values which are the images of critical points of index 0, index 1, index 1, and index 2 in this order. The inverse images of regular values varies as the empty set, one disk, two disks, one disk, and   
the empty set in this order. See the left figure in Figure \ref{fig:torusPlanes}.
For the solid torus $M'$ which is obtained as the solid of revolution obtained from a disk in $\{(x,0,z)\in \RR^3\ \big|\ x>0\}$ by rotating around the $z$-axis, the orthogonal projection to the $z$-axis is not the Morse function. For the solid torus $M''$ obtaind by tilting $M'$ a little,  
the orthogonal projection to the $z$-axis is the Morse function. This Morse function has 4 critical values which are the images of critical points of index 0, index 1, index 1, and index 2 in this order, as before, but  the inverse images of regular values varies as the empty set, one disk, one annulus, one disk, and   the empty set in this order. 
See the right figure in Figure \ref{fig:torusPlanes}.
For these two Morse height functions, their Reeb graphs $\mathcal{R}_{F|\bd M}\lra\mathcal{R}_F$ ($\mathcal{R}_{F|\bd M''}\lra\mathcal{R}_F$) and the inverse images are illustrated in Figure \ref{fig:torus} on the left and on the right.\\

\begin{figure}
\begin{center}
\includegraphics[height=4.cm]{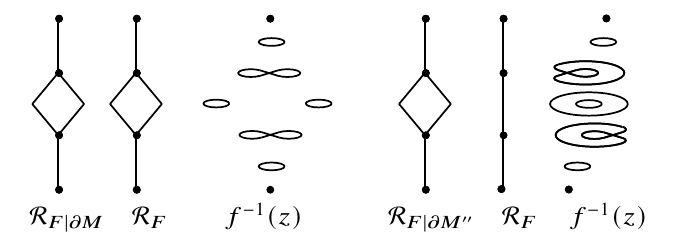}
\caption{Two Morse height functions on a solid torus, their Reeb graphs $\mathcal{R}_{F|\bd M}$ ($\mathcal{R}_{F|\bd M''}$), $\mathcal{R}_F$ and the inverse image $f^{-1}(z)$.}\label{fig:torus}
\end{center}
\end{figure}

For the Morse height function $F:M\lra \RR$ and a local minimum point $(x_0,y_0, z_0)$ of $F$, it is also a local minimum point of $F|\bd M$ and $F^{-1}(z)$ for $z>z_0$ near $(x_0,y_0, z_0)$ is a 2-disk. As $z$ increases, this 2-disk grows and the next critical point appears. If there appears a critical point in the in the interior of the disk, the critical point is a local minimum point of $F|\bd M$ and after passing through there, the component of $F^{-1}(z)$ becomes an annulus. If there appears one critical point on the boundary, the critical point is index 1 and the disk meets some other level set component and after passing through there, the topology is the component is the same as the other component. Here the number of the components of the level set decreases by 1. If there appears two critical points simultanously on the boundary, the disk meets with itself and after passing through there, either the disk splits into two disks or 
the component of $F^{-1}(z)$ becomes an annulus.
In this way, we see how the level set $F^{-1}(z)$ varies and obtain information on the shape of $M$. 

This process can be understood by the map $\mathcal{R}_{F|\bd M}\lra\mathcal{R}_F$ shown in Figure \ref{fig:bifurcate}. We should also notice that the figures illustrated upside down appear representing the Morse height functions made upside down. 

In the left figure of Figure \ref{fig:bifurcate}, in looking upwards, a disk splits into two disks. In the center figure and in the right figure of Figure \ref{fig:bifurcate}, in looking upward, a disk deforms to an annulus. 
The case where there appears one critical point on the boundary can be seen by the left figure of Figure \ref{fig:bifurcate} upside down.

\begin{figure}
\begin{center}
\includegraphics[height=4.5cm]{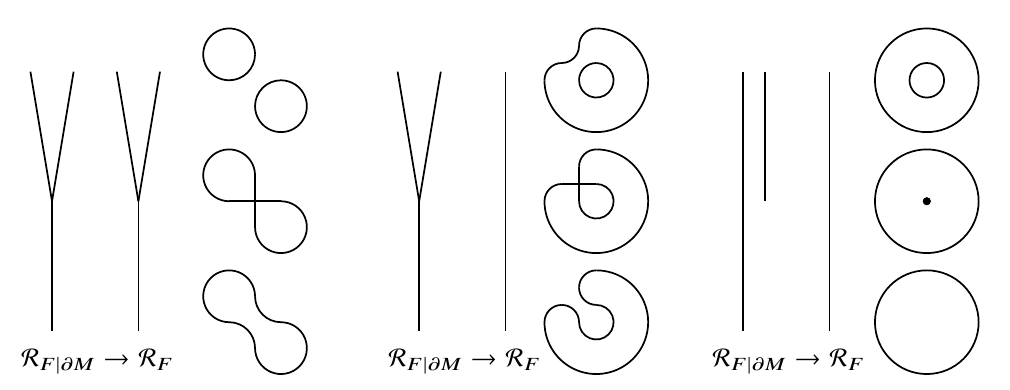}
\caption{$\mathcal{R}_{F|\bd M}\lra\mathcal{R}_F$ gives the information of the shape.}\label{fig:bifurcate}
\end{center}
\end{figure}

In the left figure of Figure \ref{fig:bifurcate}, if one of the two disks on the top disappears at a higher level, then the corresponding edges of $\mathcal{R}_{F|\bd M}$ and $\mathcal{R}_F$ stop at the level and the bounded domain is isotopic to that obtained by cutting off the corresponding 3-disk.

In Figure \ref{fig:bifurcate}, if we change the disks in the bottom of the three figures to an annulus, $\mathcal{R}_{F|\bd M}$ is added a line from bottom to the top and  $\mathcal{R}_{F}$ is unchanged. The upper left figure and the upper right figure of Figure \ref{fig:D2-2D2} shows the critical points in this situation.

If we think of the bifurcation to give a pair of pants (a disk with two interior disks deleted), the situation is more complicated. 
In  Figure \ref{fig:D2-2D2}, we show the critical points between the level sets 
a pair of pants $P$ and an annulus $A$, and between 
the level sets a pair of pants $P$ and the union of two annuli $2A$.

\begin{figure}
\begin{center}
\includegraphics[height=6.5cm]{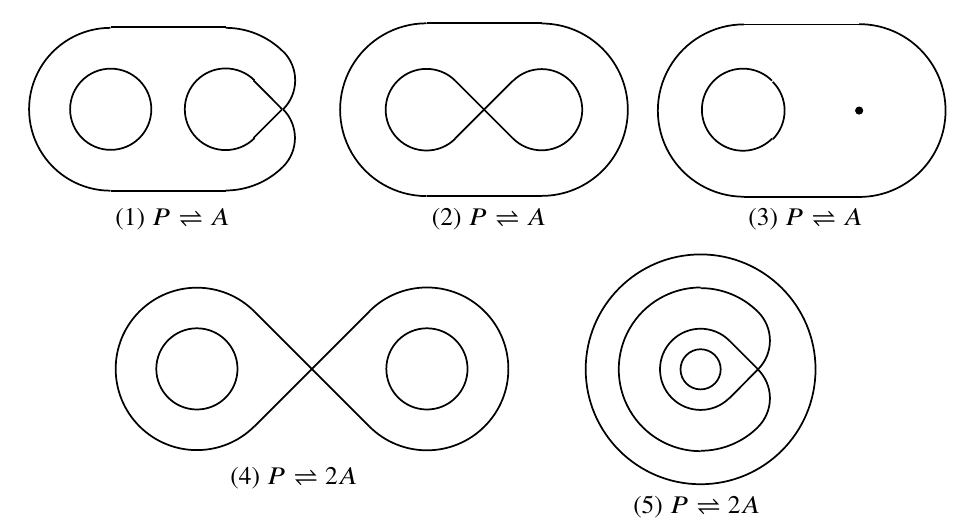}
\caption{The inverse images of the critical values which gives a pair of pants.}\label{fig:D2-2D2}
\end{center}
\end{figure}

\section{Handlebodies}
\label{sec:handlebodies}

A handlebody $M$ is isotopic to $M'$ with the Morse height function $F:M'\lra \RR$ such that $F^{-1}(z)$ is always the union of disjoint disks except for the critical values. The converse holds, that is, we have the following proposition.

\begin{proposition}
For a bounded domain $M$ and the Morse height function $F:M\lra \RR$,
if $F^{-1}(z)$ is always the union of disjoint disks except for the critical values, then $M$ is diffeomorphic to a handlebody and $M$ 
is a regular neighborhood of a spacial graph. 
\end{proposition}

\begin{proof}
This graph is of course the Reeb graph $\mathcal{R}_F$ of $F:M\lra \RR$. For the Morse function $F$ such that $F^{-1}(z)$ is always the union of disjoint disks except for the critical values, the point of the Reeb graph corresponding to a critical point is either an endpoint or a trivalent point of $\mathcal{R}_F$.
The endpoints correspond to the local maxima and local minima, and the trivalent points correspond to the critical points of $F|\bd M$ of index 1.
Then we obtain the Reeb graph as a graph with univalent or trivalent vertices corresponding to critical points. We can embed $\mathcal{R}_F$ in $M$ in a way to connecting critical points by arcs. Then we can push this embedded graph in the interior of $M$ so that $M$ can be seen as a regular neighborhood of the embedded Reeb graph.
\end{proof}

\begin{remark}
We may remove a univalent vertex with the adjacent edge and adjacent trivalent vertex except when $M$ is a ball or a solid torus. 
\end{remark}

We have diagrams for spacial trivalent graphs, or spacial graphs with univalent or trivalent vertices which are drawn on a plane with crossing points like knot diagrams. In particular, with the height functions we can draw spacial Reeb graphs by braid-like diagrams with trivalent vertices  as in Y or \rotatebox[origin=c]{180}{Y}.  

As for spacial trivalent graphs, local unwinding moves (Figure \ref{fig:unwind}) do not change the isotopy class.
\begin{figure}
\begin{center}
\includegraphics[height=3.cm]{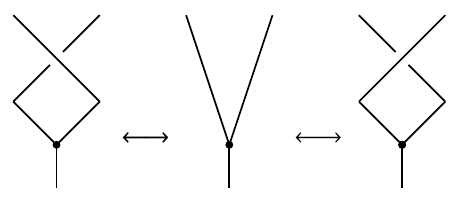}
\caption{local unwinding move}\label{fig:unwind}
\end{center}
\end{figure}

As for the handlebody embeddings represented by trivalent graphs, the Whitehead moves or IH moves  (Figure \ref{fig:WhiteheadMove}) do not change the isotopy class.
\begin{figure}
\begin{center}
\includegraphics[height=3.cm]{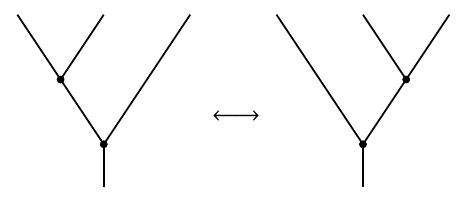}
\caption{Whitehead move or IH move}\label{fig:WhiteheadMove}
\end{center}
\end{figure}

\begin{proposition}
If $F^{-1}(z)$ is always the union of at most 3 disjoint disks except for the critical values, then $M$ is isotopic to an unknotted handlebody.
\end{proposition}

\begin{proof}
We look at the shape of $F^{-1}([\min\{F(M)\},z])$  for $z$ from $\min\{F(M)\}$ to $\max\{F(M)\}$. If $F|\bd M$ does not have critical points of index 1, then 
$\bd M$ is a 2-sphere and 
$M$ can only be a 3-disk by Theorem \ref{th:Alexander}. 

For a critical point of index 1 of $F|\bd M$ of lowest value, if it corresponds to the trivalent vertex of \rotatebox[origin=c]{180}{Y}, it can be removed with an lower edge and a univalent vertex, that is, we can modify the handlebody by isotopy to remove the critical point of index 1. Thus it corresponds to the trivalent vertex of Y. 

If $F^{-1}(z)$ is always the union of at most 2 disjoint disks except for the critical values, then for the second lowest critical point of index 1 of $F|\bd M$. Then it is necessary that the lower strands of \rotatebox[origin=c]{180}{Y} correspond to the the upper strands of the lowest Y,
and the circle formed by the two strands can be made planer by the local unwinding moves (Figure \ref{fig:unwind}). We continue this argument from the lower critical points of index 1 of $F|\bd M$, we see that  $M$ is isotopic to an unknotted handlebody.

Now we return to the case where $F^{-1}(z)$ is always the union of at most 3 disjoint disks except for the critical values. 

We look at the second lowest critical point of index 1 of $F|\bd M$.
If it correoponds to the trivalent vertex of \rotatebox[origin=c]{180}{Y}, then either the lower strands of \rotatebox[origin=c]{180}{Y} correspond to the the upper strands of the lowest Y, or one of the lower strands of \rotatebox[origin=c]{180}{Y} goes to a local minimum. In the latter case, the trivalent vertex of \rotatebox[origin=c]{180}{Y}, it can be removed with an lower edge and a univalent vertex as before. In the former case by shortening the third strand to upper vertex if necessary, the circle formed by the two strands can be made planer by the local unwinding moves (Figure \ref{fig:unwind}). Then we proceed to look at the next lowest critical point of index 1 of $F|\bd M$.

If the second lowest critical point of index 1 of $F|\bd M$ correoponds to the trivalent vertex of Y, then by using the local unwinding moves (Figure \ref{fig:unwind}) concerning the lowest critical point, the braid diagram between the lowest critical point and the second lowest critical point can be modified to the left diagram or the right diagram of Figure \ref{fig:WhiteheadMove}. Then we look at the third lowest critical point of index 1 of $F|\bd M$. This critical point necessarily correoponds to the trivalent vertex of \rotatebox[origin=c]{180}{Y}.
We have a braid of 3 strands between the second lowest critical point and the third lowest critical point. Since the braid group of 3 strands is generated by $\sigma_{12}$ and $\sigma_{23}$,  by the combination of the Whitehead moves (Figure \ref{fig:WhiteheadMove}) and the local unwinding moves (Figure \ref{fig:unwind}), we can make the braid of 3 strands trivial. Then we use the Whitehead moves to make the circular part of the graph to the simple position as in the diagram on the right of Figure \ref{fig:Lowest3Move}. 
Then we proceed to look at the next lowest critical point of index 1 of $F|\bd M$.

\begin{figure}
\begin{center}
\includegraphics[height=3.cm]{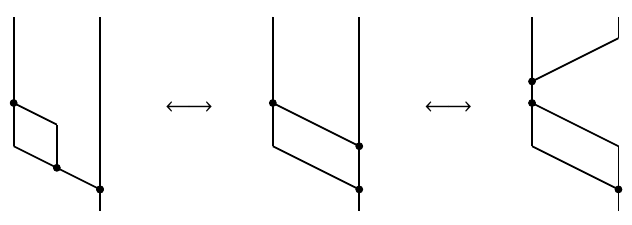}
\caption{Modify the graph at the lowest 3 vertices by the Whitehead moves}\label{fig:Lowest3Move}
\end{center}
\end{figure}

Thus we can simplify the diagram from the bottom to the top and we see that 
the resulted diagram is planer and $M$ is isotopic to an unknotted handlebody.
\end{proof}

\begin{remark}
If $F^{-1}(z)$ is the union of 4 disjoint disks for some $z$, $M$ may be knotted. For, a tubular neighborhood of a trefoil knot can have the Morse function such that $F^{-1}(z)$ is the union of at most 4 disjoint disks except for the critical values.
\end{remark}

\section{Weighted Reeb graph of the Morse height function}
\label{sec:weighted_Reeb_Graph}

We found that it is convenient to define the weighted Reeb graph $\mathcal{R}_{F}^w$ for the Morse height function $F$. 

\begin{definition}[Weighted Reeb graph] The weighted Reeb graph
$\mathcal{R}_{F}^w$ of the Morse height function $F:M\lra \RR$ is given as the Reeb graph $\mathcal{R}_{F}$ together with the nonnegative integer $1-\chi(F^{-1}(z)_i)$ for each component $F^{-1}(z)_i$ of the inverse image $F^{-1}(z)$ of regular values $z$. Here $\chi$ is the Euler characteristic number.
\end{definition}

\begin{remark} 
The connected component $F^{-1}(z)_i$ is a disk with holes and $1-\chi(F^{-1}(z)_i)$ is the number of holes. It is also the number of connected components of $\bd F^{-1}(z)_i$ minus 1.
\end{remark}

By observing the situation around a critical point of $F|\bd M$, we have the following basic fact.

\begin{proposition}\label{prop:weight}
For the Morse height function $F:M\lra \RR$ on the bounded domain $M$, the weights on the edges of the weighted Reeb graph $\mathcal{R}_F$ can change at critical values in the following way:
\begin{itemize}
\item[(1)] 
At the critical value $z$ where the number of connected components of $F^{-1}(z)$ changes, the point in the Reeb graph $\mathcal{R}_F$ causing this change is a trivalent point. Then the sum of the weights of 2 edges in the upper or lower direction is equal to that of the third edge. This corresponds to connecting or disconnecting boundary components of distinct connected components of $F^{-1}(z)$. 
\item[(2)]
At the critical value $z$ where the number of connected components of $F^{-1}(z)$ does not change, there is either a trivalent point or an endpoint of the Reeb graph $\mathcal{R}_{F|\bd M}$ at this level. 
This corresponds either to connecting or disconnecting boundary components of one connected component of $F^{-1}(z)$ or to create or erase a boundary at an interior point. 
\end{itemize}
\end{proposition}

We can put more information read from the map $\mathcal{R}_{F|\bd M}\lra\mathcal{R}_F$ on the vertices of the weighted Reeb graph $\mathcal{R}_F^w$. That is the indices of the critical points. 

\begin{definition}[weighted indexed Reeb graph]
The weighted indexed Reeb graph $\mathcal{R}_F^{wi}$ for the Morse height function $F:M\lra \RR$ is given by the weighted Reeb graph $\mathcal{R}_F^{w}$ with the vertices correoponding to the critical points  being marked with the indices of the critical points. 
\end{definition}

\begin{figure}
\begin{center}
\includegraphics[height=8cm]{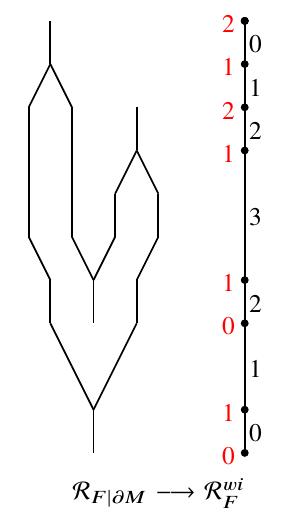}
\caption{The weighted indexed Reeb graphs for the knot complements of Figure \ref{fig:mathsoc_trefoil}.  The numbers on the edges are the weights of the edge and the numbers in red on the vertices are the indices of the vertices. }\label{fig:indexedReeb}
\end{center}
\end{figure}

\begin{figure}
\begin{center}
\includegraphics[height=13cm]{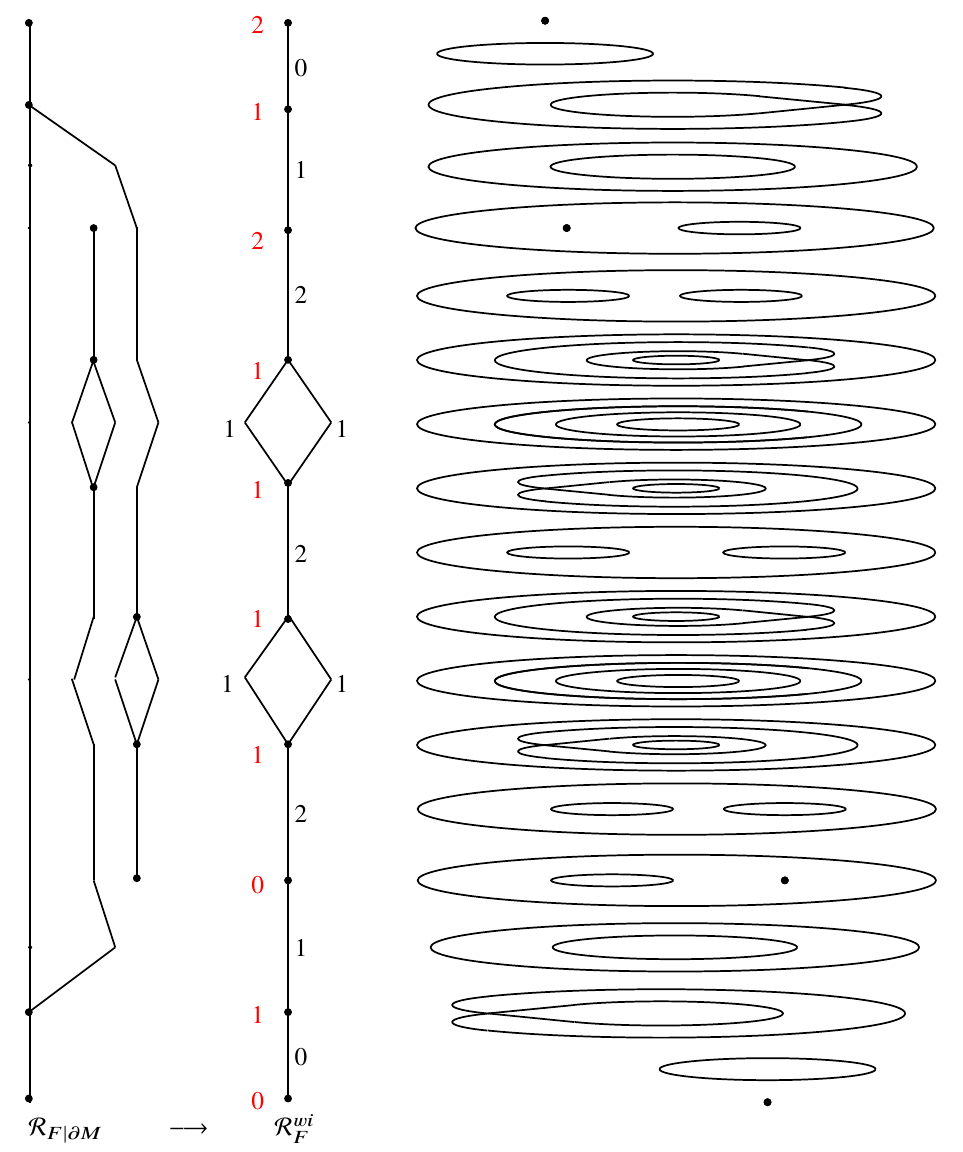}
\caption{This is a bounded domain  $M$ with the Morse height function of weight less than 3 which is diffeomorphic to the exterior of the knotted handlebody of genus 2 which is $4_1$ in the table of \cite{IKMS}.}\label{fig:counterExample}
\end{center}
\end{figure}

The indices are 1 for the trivalent vertices of the weighted indexed Reeb graph $\mathcal{R}_F^{wi}$,  the indices are 0 or 2 for the univelent vertices. These can be omitted. 
 The indices should be 0 or 1 (resp. 2 or 1) at the bivalent vertex between the lower edge with weight $k$ and the upper edge with weight $k+1$ (resp. $k-1$).
 
In Figure \ref{fig:indexedReeb}, we show the weighted indexed Reeb graphs for the knot complements of Figure \ref{fig:mathsoc_trefoil}.  The numbers on the edges are the weights of the edge and the numbers in red on the vertices are the indices of the vertices.
The two knot complements have the same Reeb graph $\mathcal{R}_{F|\bd M}$ and the same weighted indexed Reeb graph $\mathcal{R}^{wi}_F$.
Evidently the weighted indexed Reeb cannot have the information on the braiding of the cylinders in the inverse image $F^{-1}([z_0,z_1])$ of intervals $[z_0,z_1]\subset \RR$ formed by regular values.

\begin{remark}
The diffeomorphism type of the bounded domain is determined by the  
weighted indexed Reeb graph $\mathcal{R}_F^{wi}$ if the weight is less than 2.
If the weight is larger than 1, the ambiguity appears from the choice of the boundary component of the level set $F^{-1}(z)$ where the critical point of index 1 appears.
\end{remark}

We studied in the previous section \ref{sec:handlebodies}, the case where the weight of
$\mathcal{R}_F^w$ is always 0 and this implies that the bounded domain is a handlebody. 
If the weight is 0 or 1, that is, if the components of $F^{-1}(z)$ for a regular value $z$ are disks or anuli, then  we can show that $M$ is still a handlebody.

\begin{theorem}\label{th:w<2}
Let $M$ be a bounded domain with (one) boundary component $\bd M$. 
If the weights in the weighted Reeb graph
$\mathcal{R}_{F}^w$ of the Morse height function $F:M\lra \RR$ are less than 2, then $M$ is a handlebody.
\end{theorem}

\begin{proof} 
In the weighted Reeb graph $\mathcal{R}_{F}^w$, the edges of weight 1 form a disjoint union of segments where $F$ is monotone. These segments correspond to thickened cylinders each components of whose level sets are annuli. There are fattened arcs corresponding to edges of weight 0.  The endpoints of these segements of weight 1 should be the type (2) of Proposition \ref{prop:weight}. If both of endpoints of a segment correspond to endpoints of $\mathcal{R}_{F|\bd M}$, the family of the inner boundary of the annular component form a boundary component of $M$ and this contradicts that $\bd M$ is connected. Hence one of the endpoints of a segment corresponds to a trivalent point of $\mathcal{R}_{F|\bd M}$. If one endpoint of the segment corresponds to an endpoint of $\mathcal{R}_{F|\bd M}$ and the other corresponds to a trivalent point of $\mathcal{R}_{F|\bd M}$, then the part corresponding to the segment is 3-ball. For, the disk which appears corresponding to an endpoint of $\mathcal{R}_{F|\bd M}$ caps the thickened cylinder. If both of endpoints of a segment correspond to trivalent points of $\mathcal{R}_{F|\bd M}$, the thickened cylinder continues as fattened arcs both upward and downward. Then the thickened cylinder part is isotopic to a solid torus. Thus in total, $M$ is diffeomorphic to a handlebody. 
\end{proof}

When the weights of the weighted Reeb graph $\mathcal{R}^w_F$ are less than 3 but have value 2, that is, a connected component of a fiber is a disk with 2 holes, we have a bounded domain which is not a handlebody which is shown in Figure \ref{fig:counterExample}. The bounded domain  $M$ in Figure \ref{fig:counterExample} has a Morse height function of weight less than 3 which is diffeomorphic to the exterior of the knotted handlebody of genus 2 which is $4_1$ in the table of \cite{IKMS} If the complement of a handlebody knot is a handlebody, the handlebody knot is unknotted (\cite{Waldhausen}). Hence the bounded domain  $M$ is not a handlebody.

So we restrict our attentions to the case where $\bd M$ is a torus in the next section.

\section{Reeb graphs of knot exteriors}

We look at a bounded domain $M$ in the 3-dimensional Euclidian space with the boundary $\bd M$ being a torus.
Then by the Solid torus theorem \ref{th:solid_torus_theorem}, $M$ is isotopic either to a tubular neighborhood of a knot or to a knot exterior.
If $M$ is a small tubular neighborhood of a knot, there is a Morse height function $F:M\lra \RR$ such that the level sets of the regular values are disjoint union of 1-disks.

For the Reeb graphs $\mathcal{R}_{F|\bd M}$ and $\mathcal{R}_{F}$, we have the following proposition.

\begin{proposition}\label{prop:pi1_torus}
Let $M$ be a bounded domain $M$ in the 3-dimensional Euclidian space with the boundary $\bd M$ being a torus. For a Morse height function $F:M\lra \RR$, the following holds.
\begin{itemize}
\item[\rm(1)]\ 
$\pi_1( \mathcal{R}_{F|\bd M})\cong \ZZ$.
\item[\rm(2)]\ 
$\pi_1( \mathcal{R}_{F})$ is either trivial or isomorphic to $\ZZ$.
\end{itemize} 
\end{proposition}

\begin{proof}
(1) By Remark \ref{prop:pi1},  $\pi_1(\bd M)\lra \pi_1(\mathcal{R}_{F|\bd M})$ is surjective. Since $\bd M$ is a torus, $\pi_1(\bd M)\cong \ZZ^2$. Since $\mathcal{R}_{F|\bd M}$ is 1-dimensional, $\pi_1( \mathcal{R}_{F|\bd M})$ is either $0$ or $\ZZ$. 
If $\pi_1( \mathcal{R}_{F|\bd M})=0$, then $\mathcal{R}_{F|\bd M}$ is a tree and we see that $\bd M$ should be diffeomorphic to a 2-sphere. Thus $\pi_1( \mathcal{R}_{F|\bd M})\cong \ZZ$.

(2) Again, by Proposition \ref{prop:pi1},  $\pi_1(M)\lra \pi_1(\mathcal{R}_F)$ is surjective. Hence 
 $H_1(M;\ZZ)\lra H_1(\mathcal{R}_F;\ZZ)$ is also surjective. Since $\bd M\cong T^2$, by the Solid torus theorem, $M$ is either a tubular neighborhood of a knot or a knot exterior. Hence $H_1(M;\ZZ)\cong \ZZ$. Since $\mathcal{R}_F$ is 1-dimensional, $H_1(\mathcal{R}_F;\ZZ)$ is either $0$ or $\ZZ$.
Thus $\mathcal{R}_F$ is either a tree or homotopy equivalent to a circle. In the latter case it is the union of the circle and edges attached to points of the circle.
\end{proof}

\begin{proposition}\label{prop:RF_circle}
Let $M$ be a bounded domain $M$ in the 3-dimensional Euclidian space with the boundary $\bd M$ being a torus. For a Morse height function $F:M\lra \RR$, if $\pi_1( \mathcal{R}_{F})\cong \ZZ$ and the weight of the weighted Reeb graph $\mathcal{R}_F^w$ is less than 3, then $M$ is isotopic to a tubular neighborhood of a knot. 
\end{proposition}

\begin{proof}
By the Solid torus theorem, $M$ is either a solid torus which is a tubular neighborhood of a knot or a knot complement.

If $\pi_1( \mathcal{R}_{F})\cong \ZZ$, $\mathcal{R}_F$ is the union of the circle and edges attached to points of the circle. For this circle on $\mathcal{R}_F$, we can find a closed curve $c$ in $M$ which is transverse to level sets over the regular value levels and which maps to the circle. 
We look at the level set $F^{-1}(z)$ of a regular value $z$, such that one of the connected components of $F^{-1}(z)$ intersects $c$ in $M$.

If there is an edge of weight 0 in the circle in the Reeb graph $\mathcal{R}_F$, the inverse image of a point of the edge or weight 0 is a compressing disk of $M$ and $M$ is a solid torus.

Assume that the circle in the Reeb graph $\mathcal{R}_F$
 consists of the edges of weight 2 and weight 1. Then 
there is a connected component of a level set which is an annulus which intersects $c$.

If this annulus is essential, then $M$ is the complement of a composite knot. Moreover the boundary of the annulus is the boundary of two meridian disks of the tubular neighborhood of the knot and the annulus and the two disks form a sphere which divides the 3-sphere to two 3-disks each of which  contains a non trivial knot.
Since this annulus divides the knot complement $M$ into two parts, the annulus is the boundary of one part and the algebraic intersection number of any circle in $M$ and the annulus is zero. This contradicts the choice of the level surface and the annulus should be boundary compressible. 
That is, there is a compressing disk whose bundary is the union of a boundary arc and an arc on the annulus joining two boundary circles.

If there is a boundary compressing disk to the annulus. we have an isotopy which pushes the arc on annulus to the boundary along the disk. In this way, we obtain an embedded disk which is the union of two (positive and negative) parallel disks of the boundary compressing disk and the annulus with a tubular neighborhood of the arc deleted. Since the intersection number of this embedded disk and circle is 1, the boundary of the disk is essential on the boundary. Thus this disk is compressing disk and $M$ is a solid torus.
\end{proof}

\begin{theorem}\label{th:w<3knotExt}
Let $M$ be a knot exterior. If the weights in the weighted Reeb graph
$\mathcal{R}_{F}^w$ of the Morse height function $F:M\lra \RR$ are less than 3, then $M$ is a solid torus.
\end{theorem}

\begin{proof} 
By Proposition \ref{prop:pi1_torus}, the Reeb graph $\mathcal{R}_F$ has the homotopy type of one point or a circle. By Proposition \ref{prop:RF_circle}, the theorem is true if $\mathcal{R}_F$ has the homotopy type of a circle.

We assume that the Reeb graph $\mathcal{R}_F$ has the homotopy type of one point, that is,  $\mathcal{R}_F$ is a tree.
Then the union of edges of weight 1 and 2 of $\mathcal{R}^w_F$ form a tree as well and $\mathcal{R}^w_F$ can be simplified so that edges of weight 0
are added to the free edges of weight 1.

We look at the weighted Reeb graph $\mathcal{R}^w_F$ of the Morse height function $F:M\lra \RR$. We are looking at the case where the weights are less than 3.
Then the edges of weight 2  form a disjoint union of segments where $F$ is monotone as in the previous case, and it is always the case for the union of edges of highest weight.

If the weighted Reeb graph $\mathcal{R}_{F}^w$ has a segment consisting of edges of weights 1-2-1, then the combination of the upper and the lower singularitis are listed as follows using the numbering in Figure \ref{fig:D2-2D2}. Here {}' indicates the left-right mirror image.
\begin{center}
(1)-(1), (1)-(1)', (1)-(2), (1)-(3), (1)-(3)', \\(2)-(2), (2)-(3), (3)-(3), (3)-(3)'. 
\end{center}
We show that these cannot appear for the Morse height functions of non-trivial knot complement. 
Here trivial knot complement is the standard embedding of solid torus, and if we find a compressing disk in a knot complement, it is trivial knot complement. 
We assume that the Morse height function is chosen so that 
the number of the edges of weight 2 is minimum among the embeddings of the non-trivial knot complement. This forces that there are no union edges of weight 0 which branch out.

By Figures \ref{fig:upperlower11}, \ref{fig:upperlower12} and \ref{fig:upperlower23} with the captions there, we see that the vertical series of segments of weights 1-2-1 does not appear in the weighted Reeb graphs $\mathcal{R}_{F}^w$ of non-trivial knot complements.

\begin{figure}
\begin{center}
\includegraphics[height=3cm]{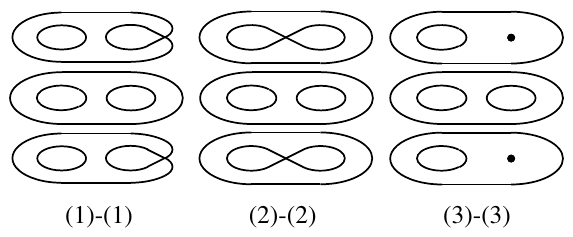}
\caption{In the cases (1)-(1) and (2)-(2), we see the compressing disks vertivally connecting the upper and the lower singularities. Their boundary circles compressed by the disks intersect the level circles, hence they are essntial. In the case (3)-(3), we find a 2-sphere as a component of boundary. 
}\label{fig:upperlower11}
\end{center}
\end{figure}
 
\begin{figure}
\begin{center}
\includegraphics[height=3cm]{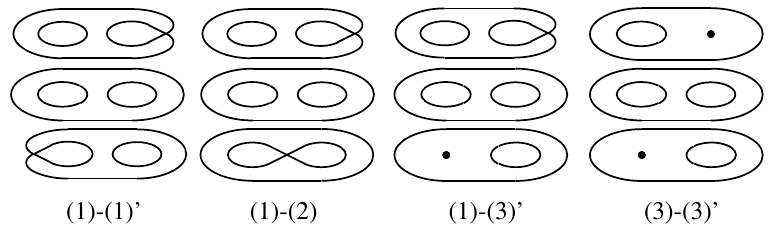}
\caption{In the cases (1)-(1)', (1)-(2), (1)-(3)' and (3)-(3)', we can change the heights of the singularities and the corresponding edges of weight 2 in the Reeb graph disappear.
}\label{fig:upperlower12}
\end{center}
\end{figure}

\begin{figure}
\begin{center}
\includegraphics[height=3cm]{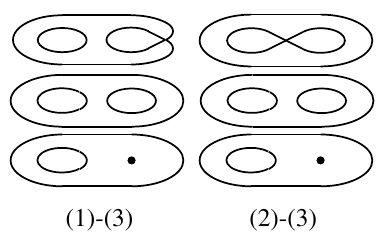}
\caption{In the cases (1)-(3) and (2)-(3), we can cancel the singularities 
by the birth-death deformation of the Morse height functions. and the corresponding edges of weight 2 in the Reeb graph disappear.
}\label{fig:upperlower23}
\end{center}
\end{figure}

\begin{figure}
\begin{center}
\includegraphics[height=6cm]{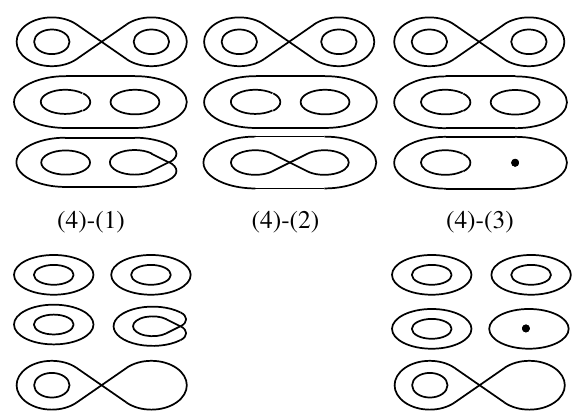}
\caption{In the cases (4)-(1) and (4)-(3), we can change the heights of the singularities and obtain the Morse function shown below (4)-(1) and (4)-(3). Then the corresponding edges of weight 2 in the Reeb graph disappear. The case (4)-(2) may happen.
}\label{fig:upperlower42}
\end{center}
\end{figure}

\begin{figure}
\begin{center}
\includegraphics[height=6cm]{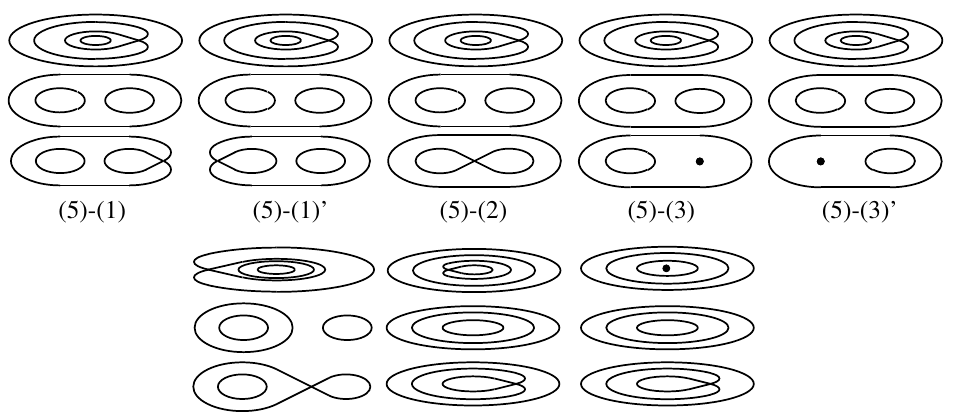}
\caption{Here are the cases (5)-(1), (5)-(1)', (5)-(2), (5)-(3) and (5)-(3)'. The hole of (5) in the center of the top level continues to the right hole of the middle level. In the cases (5)-(1)', (5)-(2) and (5)-(3), we can change the heights of the singularities and obtain the Morse function shown below (5)-(1)', (5)-(2) and (5)-(3). Then the corresponding edges of weight 2 in the Reeb graph disappear. The cases (5)-(1) and (5)-(3)' may happen.
}\label{fig:upperlower52}
\end{center}
\end{figure}

Now we look at the cases where the edge of weight 2 is adjacent to two edges of weight 1 at the upper end point and  is adjacent to one edge of weight 1 at the lower end point. There are the following cases.
\begin{center}
(4)-(1), (4)-(2), (4)-(3), (5)-(1), (5)-(1)', (5)-(2), (5)-(3), (5)-(3)'. 
\end{center}
By Figures \ref{fig:upperlower42} and \ref{fig:upperlower52} with the captions there, we see that only the cases (4)-(2), (5)-(1) and (5)-(3)' may happen.

The rest are the cases where the edge of weight 2 is adjacent to two edges of weight 1 both at the upper end point and at the lower end point. There are the following cases. See Figure \ref{fig:upperlower55}.
\begin{center}
(4)-(4), (4)-(5), (5)-(5), (5)-(5)'. 
\end{center}
In the cases  (4)-(4) and (5)-(5), we see a disk with boundary on the boundary torus. If their boundary circles are not essential on the boundary torus, they bound disks on the boundary torus. The sphere formed by these disks bound a 3-disk in the knot complement $M$. Thus the disks in  the knot complement $M$ divides $M$ into two parts and the embedding of the knot complement $M$ is given by one of the two parts. Thus in the cases  (4)-(4) and (5)-(5), the number of
edges of weight 2 in the weighted Reeb graph $\mathcal{R}_F^w$ was not minimum. Hence the cases (4)-(5) and (5)-(5)' may happen.

\begin{figure}
\begin{center}
\includegraphics[height=3cm]{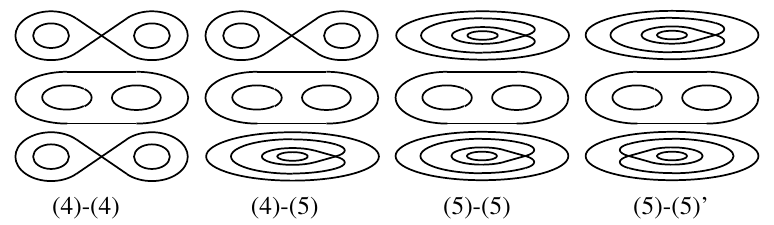}
\caption{Here are the cases (4)-(4), (4)-(5), (5)-(5) and (5)-(5)'. 
The hole of (5) in the center continues to the right hole of the middle level.
That of (5)' continues to the left hole of the middle level. Here we assert that the cases  (4)-(4) and (5)-(5) should not occur.
 }\label{fig:upperlower55}
\end{center}
\end{figure}

\begin{figure}
\begin{center}
\includegraphics[height=3cm]{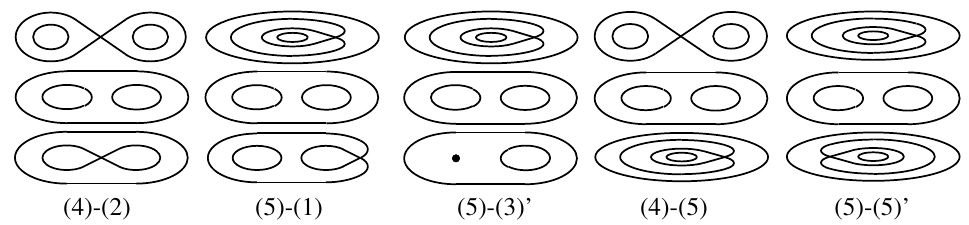}
\caption{Here are the cases (4)-(2), (5)-(1), (5)-(3)', (4)-(5) and (5)-(5)'.
 }\label{fig:upperlower5cases}
\end{center}
\end{figure}

To sum up the situation concerning the edges of weight 2 of the weighted Reeb graph $\mathcal{R}_F^w$, the possible singularities of the levels of endpoints are as follows. See Figure \ref{fig:upperlower5cases}.
\begin{center}
(4)-(2), (5)-(1), (5)-(3)', (4)-(5), (5)-(5)'.
\end{center}

In these figures, the level sets above and below continue as families of annuli and they form parts of the boundary except one family. The information on the weighted Reeb graph $\mathcal{R}^w_F$ says that the edges of weight 2 corresponding to one of these parts are connected by edges of weight 1 corresponding to thickened cylinders with family of annuli. If there is an edge of weight 0, the level set there is a disk. If this is a compressing disk, the knot complement is trivial. If its boundary circle bounds a 2-disk on the boundary torus, the 2-sphere formed by these 2-disks bounds 3-disk and we can cut off the 3-disk. Thus we may assume that an edges of weight 1 appear only as a sequence of edges weight 1 and weight 0 at the end of the weighted Reeb graph $\mathcal{R}^w_F$, which 
represent  topologically either cutting as an annular boundary or capping by a thickened 2-disk.

The next thing to do is to check whether these 5 parts can be the end parts of the weighted Reeb graph $\mathcal{R}_F^w$. The end parts means that, in the weighted Reeb graph $\mathcal{R}_F^w$, no other edges of weight 2 is on the side to ends of this part. The level sets above and below continues as families of anuli and they form parts of the boundary except one family.

Assume that the (4)-(2) piece is the end part.
If one of the upper level annuli is capped, then the (4)-(2) piece together with the cap can be deformed to a thickened cylinder.
If one of the upper level annuli is boundary, 
then the (4)-(2) piece is deformed to a thickened cylinder with hole.

Assume that the (5)-(1) piece is the end part.
If the upper interior annulus is capped, then the (5)-(1) piece together with the cap can be deformed to a thickend cylinder. If the upper interior annulus is boundary, 
then the (5)-(1) piece can be deformed to a thickened cylinder with hole.
If the upper exterior annulus is boundary and the lower annulus is boundary, then there is a vertical compressing disk.  
If the upper exterior annulus is boundary and the lower annulus is capped, 
then the (5)-(1) piece together with the cap can be deformed to a thickend cylinder with one end being boundary.
If the upper exterior annulus is capped and the lower annulus is boundary, then the (5)-(1) piece together with the cap can be deformed to a thickend cylinder with one end being boundary.
If the upper exterior annulus is capped and the lower annulus is capped, then 
the (5)-(1) piece together with the cap can be deformed to a thickened disk.

Assume that the (5)-(3)' piece is the end part.
If the upper interior annulus is capped, then the exterior circle of a level set bounds a disk and we make an edge of weight 2 disappear. 
If the upper interior annulus is boundary,  then the (5)-(3)' piece can be deformed to a thickened cylinder.
If the upper exterior annulus is boundary, then the (5)-(3)' piece can be deformed to a thickened cylinder. 
If the upper exterior annulus is capped and the lower annulus is boundary,
then the (5)-(3)' piece together with the cap can be deformed to a thickened disk.
If the upper exterior annulus is capped and the lower annulus is capped,
then there is a sphere boundary, which should not happen.

Assume that the (4)-(5) piece is the end part.
If the upper left annulus is capped, then the argument follows that of the case of (5)-(3)'. If the upper right annulus is capped, then the argument follows that of the case of (5)-(3).
If the upper left annulus is boundary and the lower exterior annulus is boundary, then we find a compressing disk. If the upper left annulus is boundary and the lower exterior annulus is capped, the (4)-(5) piece together with the cap can be deformed to a thickened thickened cylinder. 
If the upper left annulus and the lower inner annulus are boundary and .
the upper right annulus is also boundary, then we find a compressing disk.
If the upper left annulus is boundary and the lower inner annulus is capped,
then the (4)-(5) piece together with the cap can be deformed not to correspond to an edge of weight 2.
If the upper right annulus and the lower interior annulus are boundary, then the (4)-(5) piece can be deformed to a thickened cylinder with a hole. 
If the upper right annulus is boundary and the lower interior annulus is capped,  then the (4)-(5) together with the cap can be deformed to a thickened cylinder. We can verify that we examined all the necessary cases.

Assume that the (5)-(5)' piece is the end part.
If the upper inner annulus is capped, then then the argument follows that of the case of (5)-(3)'. By symmetry the case where the lower inner annulus is capped is also treated as well. If the the upper inner annulus and the lower inner annulus are boundary, then we can find a compressing disk. If the upper inner annulus, the upper exterior annulus and the lower exterior annulus are boundary, then we can find a compressing disk. If the upper inner annulus and the upper exterior annulus are boundary and the lower exterior annulus is capped, then the (5)-(5)' piece together with the cap can be deformed to a thickened cylinder. If the upper inner annulus and the lower exterior annulus are boundary and the upper exterior annulus is capped, then the (5)-(5)' piece together with the cap can be deformed to a thickened cylinder. If the upper inner annulus is boundary and the upper exterior annulus and the lower exterior annulus are capped, then there is a sphere as a component of $\bd M$.

Thus we have shown that if the weighted Reeb graph $\mathcal{R}_F^w$ with weight less than 3 is a tree, then the bounded domain is not a nontrivial knot complement.
\end{proof}

\section{Visibility of bounded domains}
\label{sec:visibility}
Now we discuss a question related to represent graphically the bounded domain.

The question is whether one can see all of the boundary of a bounded domain from far.
More precisely, we define the visibility as follows:

\begin{definition}
Let $M$ be a bounded domain in $\RR^3$.
A point $p$ of the boundary $\bd M$ is visible if
 there is a half line starting at $p$ which intersects $M$ only at $p$. 
We say $M$ is visible if all points of $\bd M$ are visible. 
\end{definition}

Visibility is not at all stable under isotopy of the bounded domain.
Thus we ask whether a bounded domain is isotopic to a visible bounded domain. 
If the bounded domain is convex, all the boundary points are visible, and for any direction any point is visible in this direction or in the opposite direction. Of course, it becomes difficult to ensure the visibility if the bounded domain $M$ has concave points, and this is the reason why visibility is not at all stable under isotopy.

For a solid torus, if it is the $\varepsilon$-neighborhood of a smooth knot, it is visible.
It is also true that a sufficiently small neighborhood of a spacial graph is visible. Thus embeddings of handlebodies can be isotoped to visible handlebodies.

As we discussed in the last section, if there is the Morse height function $F:M\lra \RR$ such that all inverse images $F^{-1}(z)$ is the union of disjoint disks, $M$ is diffeomorphic to a regular neighborhood of a spacial graph. Thus in this case  $M$  is isotopic to a visible bounded domain.

By the solid torus theorem (Theorem \ref{th:solid_torus_theorem}), the bounded domains $M$ in $\RR^3$ with $\bd M^3=T^2$  are either a tubular neighborhood of a knot in $\RR^3$ or a knot exterior, the complement in $S^3$ of a tubular neighborhood of a knot in $S^3$.

We will use Morse height functions to investigate the visibility question of bounded domains.

\begin{definition}[Minimum number of critical points $\min\#C$]
Let $M$ be a bounded domain. In the isotopy class of the embeddings of $M$ in $\RR^3$, there is an embedding $\iota$ such that the $z$ coordinate is a Morse height function on $\iota(M)$ and the number of the critical points of $z:\iota(M)\lra \RR$ is minimum in 
the isotopy class. We define $\min\#C([M])$ to be this minimum for the isotopy class $[M]$ of $M\subset \RR^3$.
\end{definition}

\begin{proposition}\label{prop:complement}
Let $M$ be a bounded domain with the Morse function minimizing the number of critical points. Then its complementary bounded domain $M^c=\RR^3\cup \{\infty\} \ssm M$ has the same minimum number of critical points as $M$.
\end{proposition}
\begin{proof}
We lool at $S^3=\RR^3\cup\infty$ with the height function $S^3\subset\RR^4 \lra \RR$ to an axis. Then it has 1 maximum and 1 minimum and other level sets are 2-spheres. 
Let the Morse function $F:M\lra \RR$ is realized as the height function $$M\subset S^3\subset\RR^4 \lra \RR$$ such that the maximum and the minimmum on $M$ is at the  North pole and the South pole of $S^3\subset\RR^4$. Then the complement $M^c$ of $M$ has  the height function $$M^c\subset S^3\subset\RR^4 \lra \RR.$$
These two height functions have the same critical points in $\bd M=\bd M^c$ and their indices coincide. The convexity and concavity of critical points of indices 0/2 alternate except the maximum at the North poke and the minimum at the South pole. Thus $\min\#CT([M])=\min\#CT([M^c])$. 
\end{proof}

\begin{remark}
In $\RR^3$ this level spheres between the maximum point and the minimum point of a Morse height function $F$ is drawn as a conformal image of level spheres of 
the projection to an axis $S^3\subset \RR^4\lra \RR$. Taking the 2-sphere with the diameter being the segement connecting the maximum point and the minimum point and if the center of this sphere is in the interior of $M$ then the image by the inversion with respect to the 2-sphere gives the complement $M^c$ as a bounded domain. We can use an isotopy along level 2-spheres to make the center of the sphere is in the interior of $M$.
\end{remark}

The minimum number of critical points among the Morse functions on a closed connected 2 manifold $S$ only depends on the topology of $S$, and it is equal to $4-\chi(S)$, where $\chi(S)$ is the Euler characteristic number of $S$. The number $\min\#C([M])$ depends on the embedding of $\bd M$ or $M$ in $\RR^3$ and for the embeddings in Figure \ref{fig:mathsoc_trefoil}\  the number is 8.

\begin{proposition}\label{prop:k-bridge}
For the tubular neighborhood of the $k$-bridge knots $M$  and the $k$-bridge knot complements $M^c$, the number is $4k$: $$\min\#CP([M])=\min\#CP([M^c])=4k.$$
\end{proposition}
\begin{proof}
For the $k$-bridge knots $K$ themselves $\min\#CP([K])=2k$ by definition.
There are $k$ maximal points and $k$ minimal points for any Morse height function on $K$ minimizing the number of critical points.  
For any tubular neighborhood of the $k$-bridge knots $M$ with the Morse height function $F$ with minimum number of critical points, we can draw the $k$-bridge knot $K$ in the interior of $M$ with minimum number of critical points among the isotopy class of $K$ in $M$. 
For each local maximum $(x_0,y_0,z_0)$ of $K$, we take the connected component of $M\cap \{z\geqq z_0\}$ and take the maximum point of $M\cap \{z\geqq z_0\}$.
This defines a map from the set of maximal points of $K$ to that of $M$. If this map is not injective, there are two maximal points on $K$ and  we have monotone increasing curves form these two points to the maximal point of $M$. These curves are in the solid torus and  one of the arcs on $K$ joining two maximal points on $K$ and the two curves to the maximal point of $M$ bounds an embedded triangle. This means we can make the number of critical points on $K$ samaller. Hence the map is injective, hence thre are at leaset $k$ maximal points of $M$.  In the same way, thre are at least $k$ minimal points. Then since the critical points are on the boundary torus, there are at least $2k$ saddle points. Hence, $\min\#CP([M])\geqq 4k$. Since we have a tubular neighborhood $M$ with the munmer of critical points being equal to $4k$, $\min\#CP([M])=4k$.
\end{proof}

In order to answer the visibility question for the moment it is necessary to 
assume the following hypothesis.

\begin{definition}[Minimum number of critical points (minNCP) hypothesis]
If a bounded domain $M$ can be isotoped to a visible bounded domain, then there is a {\it visible} bounded domain $\iota(M)$ with a Morse height function having the minimum number $\min\# C([M])$ of critical points in the isotopy class $[M]$ of $M\subset \RR^3$.
\end{definition}

It seems that visible bounded domains would not have unnecssary critical points and the minNCP hypothesis would be reasonable. However when we cancel a pair of critical points, it is not clear whether we can do among the visible embeddings. 
Under the minNCP hypothesis,
for the bounded domains whose Morse height functions necessarily have concave critical points of indices 0 and 2, we have the following proposition.

\begin{proposition}\label{prop:not_handlebody_not_visible}
Under the minNCP hypothesis,
if the bounded domain $M$ is a visible bounded domain, 
then $M$ is an embedded handlebody. In particular, under the minNCP hypothesis,
if a knot exterior is visible, then it is an unknotted solid torus.
\end{proposition} 

In order to ensure the existence of invisible point, we look at more closely the configulation of critical points of the Reeb graph $\mathcal{R}_{F|\bd M}$.

\begin{definition}
The critical points of indices 0 and 2 are either convex or concave which means the point is convex or concave as the boundary of the bounded domain $M$.  
\end{definition}

\begin{remark}
We can read the values of convex or concave critical points of index 0 or 2 from the weighted indexed Reeb graph $\mathcal{R}_F^{wi}$. 
\end{remark}

First we show that if there are no concave critical points of index 2, $M$ is an embedded handlebody.

\begin{theorem}\label{th:without_concave_critial_points}
For a connected boundary connected bounded domain $M$ in $\RR^3$, let $F:M\lra \RR$ be a Morse height function. If there are no concave critical points of index 2 of $F|\bd M$, then $M$ is an embedded handlebody. The same is true if there are no concave critical points of index 0 of $F|\bd M$.
\end{theorem}
\begin{proof}
Assume that the Morse height function $F$ has minimum number of critical points among the isotopy class of the bounded domain $M$ with no concave critical points of index 0.

We look at the minus gradient flow of the Morse function restricted to $\bd M$.
For the minus gradient flow, there are finitely many orbits which does not goes to the local minimum points. They are the local maximum points and the orbits which goes to saddle points. We may assume that the critical points of $F|\bd M$ have distinct (critical) values and there are no orbits connecting two saddles. 

We look at the change in shape of $F^{-1}((-\infty,z])$ as $z$ varies from the minimum value to the maximum value.
At the minimum value $z$ of $F|\bd M$, the critical point $(x,y,z)$ is of index 0 and convex. Hence in a neighborhood of $(x,y,z)$, $F^{-1}((-\infty,z+\varepsilon])$ for small positive $\varepsilon$ is a 3-disk, and it is contractible. 

As $z$ increases, we may have critical points of index 0, critical points of index 1 or convex critical points of index 2. We look at what happens as $z$ passes through the critical values in each case.

When $z$ passes through the critical value of a convex critical point of index 0, a 3-disk appears in $F^{-1}((-\infty,z+\varepsilon])$ 
and a contractible connected component is added. 
When $z$ passes through the critical value of a concave critical point $(x,y,z)$ of index 0, there appears a hole in $F^{-1}((-\infty,z+\varepsilon])$ with the deepest point at the critical point $(x,y,z)$, and the homotopy type of $F^{-1}((-\infty,z+\varepsilon])$ is the same as that of $F^{-1}((-\infty,z-\varepsilon])$.  $F^{-1}((-\infty,z-\varepsilon])$ is a deformation retract of $F^{-1}((-\infty,z+\varepsilon])$.

When $z$ passes through the critical value of a critical point $(x,y,z)$ of index 1, there are two cases; either the outward normal vector at $(x,y,z)$ is upward or it is downward. 

If the outward normal vector at the critical point $(x,y,z)$ is upward, the level set $F^{-1}(z-\varepsilon)$ near $(x,y,z)$ splits into two parts to the level set $F^{-1}(z+\varepsilon)$ near $(x,y,z)$. This does not change the homotopy type, that is, the homotopy type of $F^{-1}((-\infty,z+\varepsilon])$ is the same as that of  $F^{-1}((-\infty,z-\varepsilon])$. $F^{-1}((-\infty,z-\varepsilon])$ is a deformation retract of $F^{-1}((-\infty,z+\varepsilon])$.

If the outward normal vector at the critical point $(x,y,z)$ is downward, the two components of the level set $F^{-1}(z-\varepsilon)$ near $(x,y,z)$ becomes one connected component of the level set $F^{-1}(z+\varepsilon)$. This corresponds to attaching a solid 1-handle to $F^{-1}((-\infty,z-\varepsilon])$. This solid 0ne handle can be seen as a collar neighborhood of the band 1-handle attached to $(F|\bd M)^{-1}(-\infty,z-\varepsilon])$. 
The homotopy type of $F^{-1}((-\infty,z+\varepsilon])$ is that of the union of 
  $F^{-1}((-\infty,z-\varepsilon])$ and the stable manifold (arc) of the critial point $(x,y,z)$ of the gradient flow for $F|\bd M$ on $\bd M$. Thus if  $F^{-1}((-\infty,z-\varepsilon])$ has the homotopy type of 1-dimensional cellar complex then  $F^{-1}((-\infty,z+\varepsilon])$ also has the homotopy type of 1-dimensional cellar complex.  This attaching 1-handle causes either making separated components connected or forming a circle which is non-trivial in the fundamental group. 

When $z$ passes through the critical value of a convex critical point $(x,y,z)$ of index 2, then this is a maximal point and the homotopy type of $F^{-1}((-\infty,z+\varepsilon])$ is the same as that of  $F^{-1}((-\infty,z-\varepsilon])$. $F^{-1}((-\infty,z-\varepsilon])$ is a deformation retract of $F^{-1}((-\infty,z+\varepsilon])$.

These local arguments show that the bounded domain $M$ has the homotopy type of 1 dimensional cellar complex. This implies that $M$ is isotopic to a handlebody or a regular neighborhood of a spacial graph. 

More explicitly, 
the 1-dimensional cellar complex is the union of following parts:
\begin{itemize}
\item 
the stable manifolds (arc) of the critial points  $(x,y,z)$ of index 1 with downward outward normal vector of the gradient flow for $F|\bd M$ on $\bd M$.
\item
when such a stable manifolds is an unstable arc of a concave critical point $p_1$ of index 0, the vertical line segment $p_2p_1$ in $M$ with $p_1$, $p_2\in \bd M$, and the flow line of the gradient flow for $F|\bd M$ from a convex critial point $p_3$ of index 0 to $p_2$. Here we note that generically this flow line is not from a critial point of index 1.
\end{itemize}

This 1-dimensional cellar complex is a deformationretract of $M$ and $M$ is isotopic to a regular neighborhood of this 1-dimensional cellar complex.
\end{proof}

\begin{corollary}
For a bounded domain $M$, we can find several simple curves $\{c_i\}$ in $M$ joining two points of $\bd M$ so that $M$ with tubular neighborhoods of the curves $\{c_i\}$ removed is an embedded handlebody.
We can find several simple curves $\{c'_j\}$ simultanuously in the complement $M^c$ joining two points of $\bd M^c=\bd M$ 
so that $M^c$ with tubular neighborhoods of the curves $\{c'_j\}$ removed is an embedded handlebody. 
We can take the curves $\{c_i\}$ and $\{c'_j\}$ disjointly, and 
$M$ with tubular neighborhoods of the curves $\{c_i\}$ removed and
 tubular neighborhoods of the curves $\{c'_j\}$ attached is an embedded handlebody, so is $M^c$ with tubular neighborhoods of the curves $\{c'_j\}$ removed and  tubular neighborhoods of the curves $\{c_i\}$ attached.
This implies by the theorem of Waldhausen \cite{Waldhausen}, the handlebodies give the canonical Heegaard decomposition of $S^3$.
\end{corollary}
 
\begin{proof}
We use the Morse height function  $F:S^3\lra \RR$ as in Proposition \ref{prop:complement}. For the Morse height functions $F|M:M\lra \RR$ and $F|M^c:M^c\lra \RR$, if it has concave critical points of index 2, we can dig holes along the curves $\{c_i\}\subset M$ or $\{c'_j\}\subset M^c$ upward from the concave critical points of index 2 to the boundary and the resultant bounded domain becomes an embedded handlebody by Theorem \ref{th:without_concave_critial_points}. We can arrange the exits $\{\bd c_i\}$ and $\{\bd c'_j\}$ of the holes are disjoint on $\bd M=\bd M^c$. In this way we remove several 2-handles which are tubular neighborhoods of $\{c_i\}$  from $M$ or of $\{c'_i\}$  from $M^c$ to obtain embedded handlebodies. Then we attach several 1-handles $\DS\bigcup_i N(c_i)$ which are tubular neighborhoods of $\{c_i\}$ to $\DS M^c\ssm \bigcup_j N(c'_j)$, and 
we attach several 1-handles $\DS \bigcup_j N(c'_j)$ which are tubular neighborhoods of $\{c'_j\}$ to $\DS M\ssm \bigcup_i N(c_i)$. Then both 
$$\text{$\DS (M^c\ssm \bigcup_j N(c'_j))\cup \DS\bigcup_i N(c_i)$ and 
$\DS (M\ssm \bigcup_i N(c_i))\cup \DS\bigcup_j N(c'_j)$}$$ are handlebodies.
This gives a Heegaard splitting of $S^3$ and by the result by 
the result of Waldhausen \cite{Waldhausen}, the handlebodies give the canonical Heegaard decomposition of $S^3$.
\end{proof}

\begin{remark}
The minimum number of 2-handles in $M$ such that $M$ with the 2-handles removed is an embedded handlebody is called the tunnel number and studied mainly for the knot complements. For the knot complement, as is well konwn and as we saw in Proposition \ref{prop:k-bridge}, the tunnel number is less than the bridge number.
\end{remark}

\vskip5mm

\textit{Proof} {\textit{of Proposition \ref{prop:not_handlebody_not_visible}.}}
Let $M$ be a bounded domain which is not an embedded handlebody. We assume that $M$ is visible. 
By Theorem \ref{th:without_concave_critial_points}, a Morse height function $F:M\lra \RR$ should have concave critical points of index 0 and index 2.
Under the minNCP hypothesis, we may assume that $M$ is visible and the number of critical points is minimal among the isotopic embeddings of $M$ in $\RR^3$, which means the number of critical points cannot be made smaller by an isotopy.

We take a concave critical point $(x_0,y_0,z_0)$ of index 0 and we look at the Morse function $F|\bd M$. For the heights in $(z_0,z_0+\varepsilon)$ for some small positive $\varepsilon$, we see the concentric circles developing from $(x_0,y_0,z_0)$. If these circles goes up to a critical point of index 2, this part of boundary is a 2-sphere and this will not happen.
Hence, as we increase $z$, these circles meet the first critical point of index 1 (or 2).

\begin{figure}
\begin{center}
\includegraphics[height=4cm]{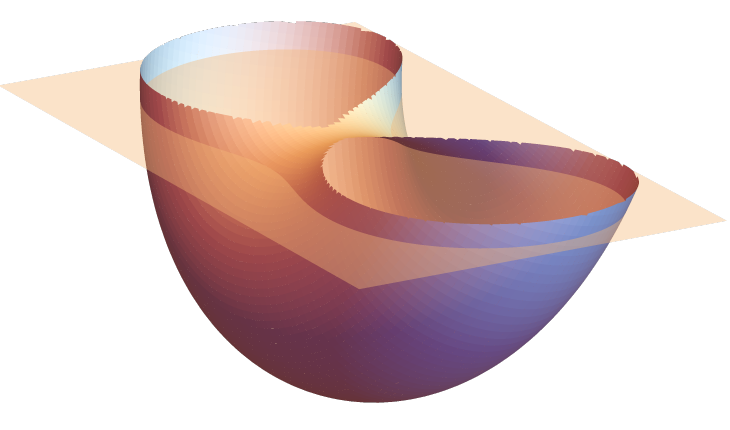}

\includegraphics[height=4cm]{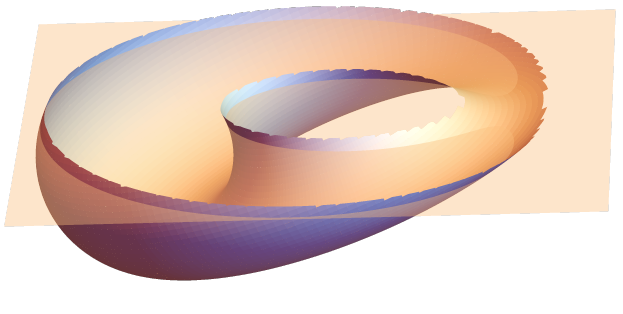}

\includegraphics[height=4cm]{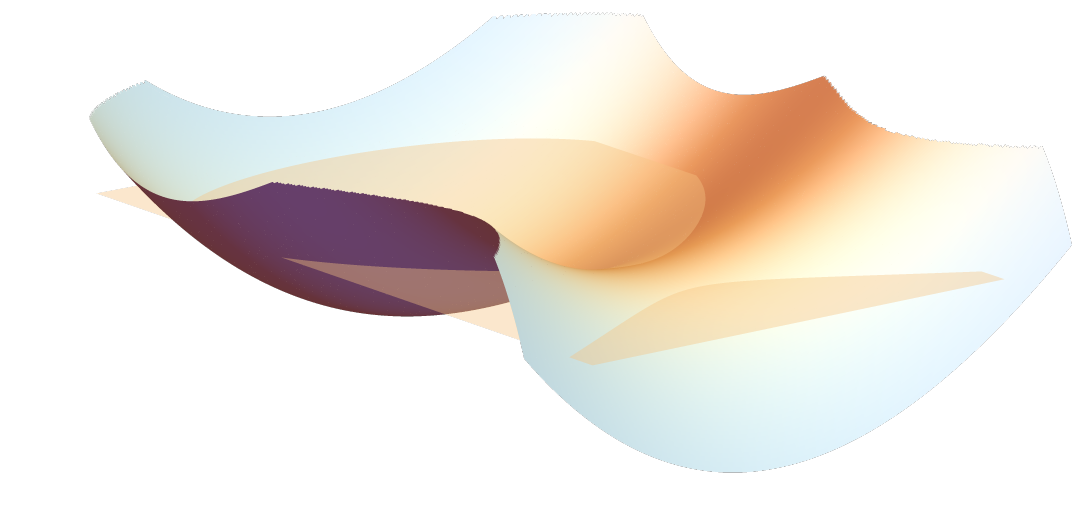}

\caption{The cases (1), (2), (3) in the 5 cases.
 }\label{fig:3cases}
\end{center}
\end{figure}

There are 5 cases: See Figure \ref{fig:3cases}.
\begin{itemize}
\item[(1)]\ 
 two points of the circle comes closer in the interior of the circle.
\item[(2)]\ 
 two points of the circle comes closer in the exterior of the circle.
\item[(3)]\ 
 another circle in the exterior of the concentric circles comes closer and meet at a point.
\item[(4)]\ 
 another circle in the interior of the concentric circles comes closer and meet at a point.
\item[(5)]\ the circle converges to a critical point of index 2.
\end{itemize}

In the case (1), the critical point $(x_1,y_1,z_1)$ of the next height is of index 1 and the outward normal vector is downward.  This implies that  cannot be seen. This is because the tangent plane at $(x_1,y_1,z_1)$ is horizontal and 
the point $(x_1,y_1,z_1)$ can only be seen from below of the horizontal tangent plane. Since the union of concentric circles and the horizontal tangent plane bound a 3-disk, if we consider the part below $z_1+\varepsilon$, the point $(x_1,y_1,z_1)$ cannot be seen from the two disks on the horizontal plane of height $z_1+\varepsilon$ bounded by the two level circles. See the top figure in Figure \ref{fig:3cases}.

In the case (4), in order to have a circle in the interior of the concentric circles, we need to have a convex critical point $(x_2,y_2,z_2)$ of index 0 between the heights $z$ and $z_1$. Since the tangent plane $(x_2,y_2,z_2)$ passes through the level circle, the point $(x_2,y_2,z_2)$ cannot be seen.

In the case (3), we can modify the embedding as follows. Note that 
each concentric circle bounds a 2-disk on the horizontal plane and this 2-disk does not intersect other part of $M$, for otherwise, we should have a 
convex critical point of index 0 on some plane lower than $z_1$ and this critical point cannot be seen. Then we take the 3-disk bounded by the union of concentric circles and the horizontal plane of height $z_1$. We can use this 3-disk to isotope the basin filled with concentric circles to a disk a little higher than $z_1$ and  eliminate the pair of critical points of index $0$ and index 1, and this contradicts the minimality of the number of critical points. See the bottom figure in Figure \ref{fig:3cases}.

On the case (5), we find 2-sphere as a boundary component and this does not happen.

Thus by the visibility assumption and the minNCP hypothesis, only the case (2) happens. 

In the case (2), the critical point $(x_1,y_1,z_1)$ of the next height is of index 1 and the outward normal vector is upward. Then the level set $(F|\bd M)^{-1}(z_1+\varepsilon)$ contains two circles which bound an annulus on $\bd M$ which contains $(x_0,y_0,z_0)$ and $(x_1,y_1,z_1)$. These two circles also bound an annulus on the horizontal plane and this annulus does not intersect the other part of $M$. See the middle figure in Figure \ref{fig:3cases}. Then we look at the basin developed from this annulus. As $z$ increases, there are 5 cases and another case 
as we looked at the the basin developed from the concave critical point.
\begin{itemize}
\item[(1)]\ 
 two points of the bundary of the annulus come closer in the interior of annulus.
\item[(2)]\ 
 two points of the bundary of the annulus come closer in the exterior of the annulus..
\item[(3)]\ 
 another circle in the exterior of the annulus comes closer and meet at a point.
\item[(4)]\ 
 another circle in the interior of the annulus comes closer and meet at a point.
\item[(5)]\ 
the exterior boundary circle reaches a critical point of index 2. 

\item[(6)]\ 
the interior boundary circle reaches a critical point of index 2. 
\end{itemize}

In the cases (1) and (4), we find a critical point of $F|\bd M$ which cannot be seen.

In the case (3), as in the previous case (3), we can eliminate critical points. Let $(x_2,y_2,z_2)$ be the next critical point. 
$(F|\bd M)^{-1}(z_2-\varepsilon)$ contains two circles which bounds an annulus on $\bd M$ which contains $(x_0,y_0,z_0)$ and $(x_1,y_1,z_1)$. These two circles bound an annulus on the horizontal plane which does not intersect other part of $M$, and by using the solid torus bounded by these two annuli, we can eliminate the critical point $(x_2,y_2,z_2)$ with the concave critical point of $(x_0,y_0,z_0)$ index 0. Here we use the visibility assumption and the minNCP hypothesis.

In the case (5), the circles form a hemisphere and we should have a lower critical point in the interior of the 3-ball bounded by this hemisphere horizontal plane.

In the case (6), the height of the critical pint of index 2 can be made lower and we can cancel it with the critical point of index 1 in the anuulus. 

Thus by the visibility assumption and the minNCP hypothesis, in this case again only the case (2) happens.

In the case (2) for the boundary of annulus, for the next critical point $(x_2,y_2,z_2)$, the level set $(F|\bd M)^{-1}(z_2+\varepsilon)$ contains a union of 3 circles which bound a disk with 2 holes on $\bd M$ and on the horizontal  plane.

In this way, we can proceed to increase $z$ and look for the next critical point.  
That is, we have the level set $F^{-1}(z)$ which is a 2-disk with $k$ holes $\varSigma_{0,k}$ and look at the  higher level sets.

There always appear the following cases (1)-(6).
\begin{itemize}
\item[(1)]\ 
 two points of the bundary of $\varSigma_{0,k}$ come closer in the interior of $\varSigma_{0,k}$.
\item[(2)]\ 
 two points of the bundary of $\varSigma_{0,k}$ come closer in the exterior of $\varSigma_{0,k}$.
\item[(3)]\ 
 another circle in the exterior of $\varSigma_{0,k}$ comes closer and meet at a point.
\item[(4)]\ 
 another circle in the interior of $\varSigma_{0,k}$  comes closer and meet at a point.
\item[(5)]\ 
the exterior boundary circle of $\varSigma_{0,k}$ reaches a critical point of index 2. 

\item[(6)]\ 
an interior boundary circle of $\varSigma_{0,k}$ reaches a critical point of index 2. 
\end{itemize}

We can treat cases (1) and (3)--(6) as before, and the case (2) is left to be solved. But the case (2) will give a level $z$ with $(F|\bd M)^{-1}(z)$ being a union of $k+1$ circles which bound a disk with $k$ holes on $\bd M$ and on the level plane.
Since the number of circles on the level set $(F|\bd M)^{-1}(z)$ is bounded, 
this procedure does not continue endlessly. Thus we will have cases (1), (3)--(6), and either we find a point which cannot be seen or it is posssible to eliminate critical points.
\hfill \qed\\

\begin{remark}
The minNCP hypothesis is used only in the cases (3) and (6).
\end{remark}

For a bounded domain $M$ with a Morse height function $F:M\lra \RR$,
there are cases where we can find  a point on the boundary $\bd M$ which is not visible from the infinity. In particular for the knot complements of $k$-bridge knots we have the following proposition. Thus we have a proof that the knot complements of Figure \ref{fig:mathsoc_trefoil} should have  a point on the boundary $\bd M$ which is not visible from the infinity.

\begin{corollary}\label{cor:k-bridge}
If $M$ is a $k$-bridge knot complement and a orthogonal projection $F:M\lra \RR$ is a Morse height function with $4k$ critical points, then there is a point on the boundary $\bd M$ which is not visible from the infinity. 
\end{corollary}
\begin{proof}
As is written in Proposition \ref{prop:k-bridge}, $\min\#C([M])=4k$ and then the cases (3) and (6) would not appear. Hence there is a point on the boundary $\bd M$ which is not visible from the infinity.
\end{proof}

\section*{acknowledgement}
The author acknowledges very much for the helpful discussions in the 
Unsolved Problems Workshop 2024 of Japan Science and Technology Agency held in LecTore Hayama Shonan Kokusai Mura, September 21 -- 23, 2024,  
Unsolved Problems Workshop 2025 of Japan Science and Technology Agency held in Link Forest, Tama Center, September 13 -- 15, 2025, and in Topology and Geometry of Low-Dimensional Manifolds 2025, held in College Plaza Akita, Lecture Hall, October 29 --  November 1, 2025.
.
He also thanks Masaaki Suzuki who introduced him the results on spacial graphs and handlebody knots, and Osamu Saeki who showed him the results obtained by Ken-ichi Iwamoto and related results.

\end{document}